\documentclass[12pt]{amsart}
\usepackage[utf8]{inputenc}
\usepackage{amsmath}
\usepackage{amssymb}
\usepackage{amsthm}
\usepackage[
backend=biber,
style=alphabetic,
]{biblatex}
\usepackage{chngcntr}
\usepackage{enumitem}
\usepackage{extpfeil}
\usepackage[margin=2cm]{geometry}
\usepackage{hyperref}
\usepackage{parskip}
\usepackage{mathrsfs}
\usepackage{mathtools}

\hypersetup{
  colorlinks=true,
  linkcolor=blue,
  citecolor=blue
}

\newcommand{\R}{\mathbb R}
\newcommand{\Z}{\mathbb Z}
\newcommand{\C}{\mathbb C}
\newcommand{\PP}{\mathbb P}
\newcommand{\E}{\mathbb E}
\newcommand{\ip}[2]{\left\langle#1,#2\right\rangle}

\theoremstyle{plain}
\newtheorem{theorem}{Theorem}
\newtheorem{lemma}[theorem]{Lemma}
\newtheorem{corollary}[theorem]{Corollary}

\theoremstyle{definition}
\newtheorem{definition}[theorem]{Definition}
\newtheorem{example}[theorem]{Example}
\newtheorem{remark}[theorem]{Remark}

\numberwithin{theorem}{section}

\newcommand{\details}[1]{}
\newcommand{\intuition}[1]{}

\title[Time-Inhomogeneous Random Walks]{Time-Inhomogeneous Random Walks on Finite Groups and Cokernels of Random Integer Block Matrices}

\author{Elia Gorokhovsky}
\address{Harvard University, Cambridge, USA}

\date{\today}

\begin{document}

\begin{abstract}
    We study time-inhomogeneous random walks on finite groups in the case where each random walk step need not be supported on a generating set of the group. When the supports of the random walk steps satisfy a natural condition involving normal subgroups of quotients of the group, we show that the random walk converges to the uniform distribution on the group and give bounds for the convergence rate using spectral properties of the random walk steps. As an application, we use the moment method of Wood to prove a universality theorem for cokernels of random integer matrices allowing some dependence between entries.
\end{abstract}

\maketitle

\section{Introduction}

The work in this paper is motivated by a question in the theory of integer random matrices but is of independent interest to the study of random walks on groups.

A random walk on a group $G$ is a Markov chain on $G$ whose state after the $n$th step is a product $X_1X_2\dots X_n$ for independent random elements $X_i \in G$. Random walks on finite groups are well-studied in the time-homogeneous, ergodic regime, where the $X_i$ are all drawn from a fixed distribution supported on a generating set of $G$. Such random walks are known to converge to the uniform distribution $\pi$ on $G$ exponentially quickly. Namely, if we denote by $\nu_n$ the distribution of $X_1X_2\dots X_n$, then
\[
||\nu_n - \pi||_{L^2} \leq \sigma^n,
\]
where $\sigma$ is the second-largest singular value of the Markov operator of the random walk and $||\nu_n - \pi||_{L^2} \coloneqq \sqrt{\sum_{g \in G}(\nu_n(g) - |G|^{-1})^2}$ denotes the $L^2$ distance between the measures $\nu_n$ and $\pi$ viewed as functions $G \to \R$. See \cite{saloff-coste2004book} for an excellent review of these kinds of walks. 

The above inequality comes from looking at norms of convolution operators on the space of signed measures on $G$. If $X$ and $Y$ are random elements of $G$ distributed according to $\mu$ and $\nu$ respectively, then $XY$ is distributed according to the convolution \[
(\mu * \nu) (g) = \sum_{h \in G} \mu(h)\nu(h^{-1}g).
\] 
In particular, if the $X_i$ are distributed according to $\mu$, then $\nu_n$ is the $n$-fold convolution $\mu^{*n} = \underbrace{\mu * \dots * \mu}_{n\text{ times}}$. Since $\pi * \mu = \pi$, the difference $\nu_n - \pi$ can then be expressed as $(\mu - \pi) * \mu^{*{n-1}}$. The $L^2$ norm of this function can be bounded in terms of the operator norm of the convolution operator $*\mu$ restricted to a suitable subspace, which is related to its second-largest singular value $\sigma$.

Some of the assumptions can be relaxed; for instance, Saloff-Coste and Z\'u\~niga \cite{saloffcoste2007inhomogeneous} studied convergence of time-inhomogeneous Markov chains, including random walks on finite groups, in the case where each step of the random walk is irreducible (in particular, supported on a generating set of $G$). In that case, if we denote by $\sigma_i$ the second-largest singular value of the $i$th step,
 \[
||\nu_n - \pi||_{L^2} \leq \prod_{i=1}^n\sigma_i.
\]
The condition that each step of the random walk is supported on a generating set is crucial because if the subgroup generated by the supports of the steps is a proper subgroup of $G$, the random walk will surely stay in that subgroup. In that case, the second-largest singular value of the corresponding step is 1. Nevertheless, one may expect that if the supports of all steps taken together generate $G$, the random walk might still equilibrate to the uniform distribution on $G$.

A consequence of our first main result is the following theorem, which relaxes this ``generating'' assumption by extending part of \cite[Theorem 3.5]{saloffcoste2007inhomogeneous} to some time-inhomogeneous random walks where the probability measures driving each step need not be irreducible:

\begin{theorem}\label{thm:intro-random-walks}
Let $G$ be a finite group, and let $\mu_1, \mu_2, \dots, \mu_n$ be probability measures on $G$. For each subgroup $H$ of $G$, let $I_H = \{i \mid H = \langle \operatorname{supp} \mu_i\rangle\}$. Let $\mathcal{S}$ be a finite set of normal subgroups of $G$ such that $G = \left\langle \bigcup_{H \in \mathcal{S}} H\right\rangle$. Write $\nu_n = \mu_1 * \dots * \mu_n$. Also, for each $i$, let $\sigma_i$ be the second-largest singular value of $*\mu_i$ as an operator on $L^2(\langle \operatorname{supp} \mu_i \rangle)$. Let $\pi$ be the uniform distribution on $G$. 

If $I_H$ is nonempty for each $H \in \mathcal{S}$, we have \[
||\nu_n - \pi||_{L^2} \leq \sum_{H \in \mathcal{S}} \left(\prod_{i \in I_H} \sigma_i\right).
\]
\end{theorem}
We prove a more general version of this result in Theorem~\ref{thm:strong-random-walks}. 

In particular, if a time-inhomogeneous random walk on a finite group has steps supported on enough subgroups, then it converges to the uniform distribution on the group with an exponential rate controlled by subgroups that appear infrequently or mix very slowly. Adding more probability measures to the convolution $\nu_n$ may not improve the convergence rate, but it never makes the bound worse because convolution with a probability measure is non-expansive in the $L^2$ norm. A nice consequence of this is that $\mathcal{S}$ need not be an exhaustive list of every normal subgroup for which $I_H$ is nonempty. 

The main difference between this result and \cite[Theorem 3.5]{saloffcoste2007inhomogeneous} is that \cite{saloffcoste2007inhomogeneous} relaxes the time-homogeneity assumption for random walks but not the assumption that each step is supported on a generating set for the group. The new condition that the supports of the steps jointly generate $G$ is substantially weaker than the assumption that the support of each step generates $G$. 

The conditions of Theorem~\ref{thm:intro-random-walks} can be weakened so that not all the subgroups $H_i$ need to be normal (see Theorem~\ref{thm:strong-random-walks}), but see Example~\ref{ex:normality-necessary} for why some hypothesis on the subgroups is necessary. In this paper, we apply Theorem~\ref{thm:intro-random-walks} in the case where the ambient group $G$ is abelian, so that the normality condition on the subgroups $H$ becomes vacuous. However, we emphasize that the theorem applies also to nonabelian groups and may have some interesting implications to random walks on nonabelian groups with many normal subgroups (e.g., the quaternion group $Q_8$).

Our main interest in developing this theorem is an application to limiting distributions of cokernels of random matrices. Wood \cite[Theorem 2.9]{wood2019matrices} and Nguyen and Wood \cite[Theorem 1.1]{nguyenRandomIntegralMatrices2022} showed that cokernels of integer-valued random matrices approach a universal limiting distribution in the following sense. Let $(M_n)_{n=1}^\infty$ be a sequence with each $M_n$ a random $n\times (n + u)$ integer matrix with independent entries ($u \geq 0$). Wood showed that, under very weak conditions on the distributions of the entries of the $M_n$, the distribution of the isomorphism class of the random group $\operatorname{coker}(M_n) \coloneqq \Z^n/M_n(\Z^{n+u})$ converges weakly (i.e., at any finite collection of primes) as $n \to \infty$ to the distribution $\lambda_u$ on isomorphism classes of abelian groups defined as follows: if $Y \sim \lambda_u$ and $B$ is a finite abelian $p$-group, then \[
\PP[Y_p \cong B] = \frac{1}{|B|^u|\operatorname{Aut}(B)|}\prod_{k= u + 1}^\infty (1 - p^{-k})
\] 
independently for all primes $p$, where $Y_p$ is the $p$-part of $Y$ (i.e., $Y = \prod_p Y_p$, where the product ranges over all primes). If further $u \geq 1$, then $\lambda_u$ is supported on isomorphism classes of finite abelian groups, and for finite abelian $B$ we have
\[
\lambda_u(B) = \frac{1}{|B|^u |\operatorname{Aut}(B)|}\prod_{k= u + 1}^\infty \zeta(k)^{-1},
\]
where $\zeta$ denotes the Riemann zeta function. Nguyen and Wood weakened the conditions on the entries and showed strong (pointwise) convergence to $\lambda_u$. The phenomenon that the limiting distribution of $\Z^n/M_n(\Z^{n+u})$ is rather insensitive to the distributions of the entries of $M_n$ is an example of \textit{universality}. The distributions $\lambda_0$ and $\lambda_1$ are known as the Cohen-Lenstra distributions, and are conjectured to describe the distributions of class groups of imaginary and real quadratic number fields, respectively.

In her 2022 ICM talk, Wood \cite[Open Problem 3.10]{woodProbabilityTheoryRandom2023} asks if the universality class of $\lambda_u$ can be extended to cokernels of matrices with some dependent entries. There are a few specific results in this direction. Most recently, Nguyen and Wood \cite[Theorem 1.1]{nguyenRandomIntegralMatrices2022} show that the distribution $\lambda_1$ is universal for Laplacians of Erd\H{o}s-R\'enyi random directed graphs. M\'esz\'aros \cite{Meszaros2020} shows that $\lambda_0$ is universal for Laplacians of random regular directed graphs. Friedman and Washington \cite{FriedmanWashington+1989+227+239} show that the cokernels of the random matrices $I - M$, where $M$ is drawn at random from the multiplicative Haar measure on $\operatorname{GL}_{2g}(\Z_p)$, approach the $p$-part of $\lambda_0$ as $g \to \infty$. However, when there is \textit{too much} dependence in the entries of the random matrices, one gets different (but often related) limiting distributions, for example in the case of symmetric matrices (\cite{wood2017sandpile}), Laplacians of random regular undirected graphs (\cite{Meszaros2020}), products of independent random integral matrices (\cite{nguyen2024products}), and quadratic polynomials in Haar-random matrices (\cite{cheong2022polynomials}). 

There are more recent examples where the amount of dependence can be controlled quantitatively. For instance, M\'esz\'aros \cite{meszaros_phase_2024} shows that cokernels of Haar-random band matrices over $\Z_p$ converge to the $p$-part of $\lambda_0$ if and only if the width of the band grows fast enough. Kang, Lee, and Yu \cite{kang_random_2024} show that certain random matrices over $\Z_p$ with some entries fixed to zero have cokernels converging to the $p$-part of $\lambda_0$, but fixing too many entries to zero prevents convergence to $\lambda_0$.

It is natural to ask just how much (and what kind of) dependence is allowed between the entries of sequences of random matrices before their cokernels leave the universality class of $\lambda_u$.

The main application of Theorem~\ref{thm:strong-random-walks} in this paper is a Theorem~\ref{thm:intro-matrix-universality} below, which extends the result of \cite{wood2019matrices} to matrices with a rather general form of dependence in their rows and columns. We introduce a regularity condition on matrices, $(w, h, \varepsilon)$-balanced, in Definitions~\ref{def:balanced}~and~\ref{def:w-h-e-balanced}. Generally, it means that the matrix can be written as a block matrix where the blocks have height at most $h$, width at most $w$, are all independent, and each satisfy some regularity condition depending on $\varepsilon$. The key detail is that the blocks of the matrix may have dependent entries, as long as there is no dependence between blocks. (The $(w, h, \varepsilon)$-balanced condition is invariant under permutation of rows and columns, so one can also think of a $(w, h, \varepsilon)$-balanced matrix as a block matrix which is at most $h$ blocks tall, at most $w$ blocks wide, and such that the entries of each block are independent of each other, while some dependence between different blocks is allowed.) With this condition, we have:

\begin{theorem}\label{thm:intro-matrix-universality}
    Let $u \geq 0$ be an integer. Let $(w_n)_n, (h_n)_n$, $(\varepsilon_n)_n$ be sequences of real numbers such that $w_n = O(n^{\alpha_1})$, $h_n = O(n^{\alpha_2})$ and $\varepsilon_n = \Omega( n^{-\beta})$ for some $0 \leq \alpha_1, \alpha_2, \beta < 1$ satisfying \[2\alpha_1 + \alpha_2 < 1 - 2\beta.
    \]
    For each integer $n \geq 0$, let $M_n$ be an $(w_n, h_n, \varepsilon_n)$-balanced $n \times (n + u)$ random matrix with entries in $\Z$. Then the distribution of $\operatorname{coker}(M_n)$ converges weakly to $\lambda_u$ as $n \to \infty$. In other words, if $Y \sim \lambda_u$, then for every positive integer $a$ and every abelian group $H$ with exponent dividing $a$ we have \[
    \lim_{n\to\infty} \PP[\operatorname{coker}(M_n) \otimes \Z/a\Z \cong H] = \PP[Y \otimes \Z/a\Z \cong H].
    \]
\end{theorem}
Here $w_n = O(n^{\alpha_1})$ means that there is a constant $A$ independent of $n$ such that $w_n \leq An^{\alpha_1}$, and $\varepsilon_n = \Omega(n^{-\beta})$ means that there is a constant $B$ independent of $n$ such that $\varepsilon_n \geq Bn^{-\beta}$.  

The key idea of the proof of Theorem~\ref{thm:intro-matrix-universality} uses the moment method developed in \cite{wood2017sandpile} and \cite{wood2019matrices}. Understanding the cokernel of a random integer matrix reduces to finding the probability that each random column maps to zero under an arbitrary surjective group homomorphism $f\colon \Z^n \to G$ for an arbitrary abelian group $G$. To handle dependent columns, we treat several columns at a time and look at the induced surjection $(\Z^n)^m \to G^m$. We view the image of a random element of $(\Z^n)^m$ as a random walk in $G^m$ and apply Theorem~\ref{thm:strong-random-walks} to approximate the distribution of this image. Since the surjection $f$ is arbitrary, we have very little control over the distribution of the steps of this walk. In particular, they are almost never supported on all of $G^m$, which is why we need Theorem~\ref{thm:strong-random-walks} to handle random walk steps supported on proper subgroups. The $(w, h, \varepsilon)$-balanced condition allows us to bound the singular values of the associated convolution operators and get quantitative bounds on the error in terms of $w$, $h$, and $\varepsilon$.

This random walk model works for most surjections $f\colon \Z^n \to G$. However, there is a certain (controlled, by Lemma~\ref{lem:count-depth}) number of exceptionally pathological surjections $\Z^n \to G$. For these surjections, we use the $\varepsilon$-balanced condition to give a bound on how much each can affect the computation of the moments. This bound gets better the more independent blocks of columns we can find in the random matrix in question (in other words, the narrower each block is). As a result, there is an additional constraint on $w_n$ that forces us to restrict it more, explaining the asymmetry between $\alpha_1$ and $\alpha_2$ in the conditions.

The actual constraint that one should expect is that there is a tradeoff between regularity and block size \textit{for each individual block}. As a given block gets wider, one needs to impose a stronger $\varepsilon$-balancedness condition on it. On the other hand, smaller blocks may be allowed to be less regular without compromising the universality. A statement of this form follows directly from our proof, but we present the results with a uniform bound on block sizes and balancedness for simplicity.

There is a considerable body of literature pertaining to random matrices with \textit{complex} entries, with analogous universality results about distributions of eigenvalues. If $\{M_n\}$ is a sequence of $n\times n$ random complex matrices whose entries are independent, with appropriately normalized mean and variance, the empirical distribution of the eigenvalues of $M_n$ converges to the \textit{circular law}, which is the uniform distribution on the unit disc in $\C$ \cite{tao2010circle}. The universality of the circular law for spectra of a wide class of random complex block matrices was proved by Nguyen and O'Rourke in \cite{nguyenConcentrationRandomMultilinear2015}. However, we note that \cite{nguyenConcentrationRandomMultilinear2015} shows universality for block matrices with i.i.d. blocks of constant size, whereas our blocks are allowed to grow in size with the matrices.

\subsection{Notation and Terminology}

For a finite set $S$, we use $L^2(S)$ to denote the space of signed measures (equivalently, real-valued functions) on $S$, equipped with the norm $||f||_{L^2(S)}^2 = \sum_{s\in S} |f(s)|^2$. When the set $S$ is implicit, we write $||f||_{L^2}$ for $||f||_{L^2(S)}$. Any set map $f\colon S \to T$ defines a \textit{pushforward} map $f_*\colon L^2(S) \to L^2(T)$ by $f_*\mu(t) = \mu(f^{-1}(t))$. We say a signed measure $\nu$ is \textit{uniform} on $T \subset S$ if $\nu(t) = \nu(t')$ for $t, t' \in T$. For a point $f \in \R^n$ with the Euclidean metric and a linear subspace $W \subset \R^n$, we write $d_{L^2}(f, W)$ for the distance between $f$ and its orthogonal projection onto $W$. Note that this is equal to $\inf_{g \in W} |f - g|$. If $G$ is a finite group, any signed measure $\mu$ defines a linear \textit{convolution operator} (or, if $\mu$ is a probability measure, \textit{Markov operator}) $*\mu$ on $L^2(G)$ given by $\nu \mapsto \nu * \mu$. When we discuss the second-largest singular value of an operator, we are counting with multiplicity; for example, if the singular values of $M$ are $1, 1, 0$, then its second-largest singular value is $1$. 

For two finite or profinite groups $G, G'$, we write $\operatorname{Hom}(G, G')$ for the set of (continuous) group homomorphisms from $G$ to $G'$ and $\operatorname{Sur}(G, G')$ for the set of (continuous) surjective group homomorphisms from $G$ to $G'$. For a subset $S \subseteq G$, we denote by $\langle S \rangle$ the (closed) subgroup of $G$ generated by $S$. We refer to the identity element of a group as $e$.

A probability measure or distribution is a measure with total mass 1 (not signed). The uniform distribution on $G$ is usually denoted $\pi$ and is the measure on $G$ with $\pi(g) = 1/|G|$ for $g \in G$. We use $\PP[\cdot]$ for probability and $\E[\cdot]$ for expectation. We denote by $\operatorname{supp} \mu$ the support of a measure $\mu$. If a random variable $X$ has law $\mu$, we write $X \sim \mu$.

For a positive integer $n$, we write $[n]$ for the set $\{1, \dots, n\}$.

If $a_n$ and $b_n$ are sequences of positive real numbers, we write $a_n = O(b_n)$ if there is a constant $0 < K < \infty$ such that $a_n \leq Kb_n$ for all $n$, and we write $a_n = \Omega(b_n)$ if there is a constant $0 < K < \infty$ such that $a_n \geq Kb_n$ for all $n$, i.e., if $b_n = O(a_n)$. We write $a_n = o(b_n)$ if, for each constant $K < \infty$, there is a natural number $n_0$ such that $a_n \leq Kb_n$ when $n \geq n_0$.

\section{Random Walks}\label{sect:walks}

This section is devoted to proving the following stronger version of Theorem~\ref{thm:intro-random-walks} from the introduction:

\begin{theorem}\label{thm:strong-random-walks}
Let $G$ be a finite group and suppose we have a sequence of surjective homomorphisms \[
G = G_0 \xtwoheadrightarrow{Q_1} G_1 \xtwoheadrightarrow{Q_2} G_2 \xtwoheadrightarrow{Q_3} \dots \xtwoheadrightarrow{Q_{k-1}} G_{k-1} \xtwoheadrightarrow{Q_k} G_k = \{e\}.
\]
For $0 \leq j \leq k$, define $\tilde{Q}_j\colon G \twoheadrightarrow G_j$ by $\tilde{Q}_j = Q_j\circ Q_{j-1}\circ\dots\circ Q_1$ (so $\tilde{Q}_0 = \operatorname{id}_G$), and for $1 \leq j \leq k$ define $H_j \trianglelefteq G_{j-1}$ by $H_j = \ker Q_j$.

Let $\mu_1, \dots, \mu_n$ be probability measures on $G$. Let $\nu_n = \mu_1*\dots* \mu_n$. For each $j = 1, \dots, k$, let $I_j = \{i \mid \langle \operatorname{supp} (\Tilde{Q}_{j-1})_* \mu_i \rangle = H_j \}$. Let $\pi$ be the uniform distribution on $G$.
    
    For $i \in I_j$, let $\sigma_i$ be the second largest singular value of the $(\Tilde{Q}_{j-1})_* \mu_i$-random walk on $H_j$. If each $I_j$ is nonempty, we have \[
    ||\nu_n - \pi||_{L^2}^2 \leq \sum_{j=1}^k \frac{|G_{j-1}| - 1}{|G|}\left(\prod_{i \in I_j} \sigma_i^2\right) = \sum_{j=1}^k \frac{\prod_{i=j}^{k} |H_i| - 1}{|G|}\left(\prod_{i \in I_j} \sigma_i^2\right).
    \]
\end{theorem}

In the case where $k = 1$ and $H_1 = G$, we recover the first part of \cite[Theorem 3.5]{saloffcoste2007inhomogeneous}. We postpone the proof that Theorem~\ref{thm:strong-random-walks} implies Theorem~\ref{thm:intro-random-walks} until the end of this section.

The condition involving the sequence of surjective homomorphisms is an artifact of the inductive proof of the theorem. We are motivated by the special case where the support of each $\mu_i$ generates a normal subgroup of $G$. However, the proof proceeds by showing that the pushforward of $\nu_n$ is close to uniform in $G_k$, then in $G_{k-1}$. We use the steps $\mu_i$ with $i \in I_j$ only when analyzing the pushforward of $\nu_n$ to $G_j$, so we do not always need to make assumptions about the subgroup of $G$ generated by the support of each measure $\mu_i$. We state Theorem~\ref{thm:strong-random-walks} in full generality because this stronger form is more easily applicable to nonabelian (for example, nilpotent) groups.

The following example gives a case that is covered by Theorem~\ref{thm:strong-random-walks} but not by the weaker Theorem~\ref{thm:intro-random-walks}:

\begin{example}
    Consider the dihedral group $G = D_{2n} = \langle r, s \mid r^n = s^2 = (rs)^2 = e\rangle$ with $n > 2$. The subgroup $H_1 = \langle r\rangle$ is normal, but the subgroup $\tilde{H}_2 = \langle s\rangle$ is not normal. We have $G_1 = D_{2n}/\langle r\rangle = \Z/2\Z$ generated by the image of $s$, so the image of $\langle s\rangle$ in the quotient is normal. Let $Q\colon G \twoheadrightarrow \Z/2\Z$ be the quotient map. Let $\mu$ be a measure on $\langle r\rangle$ with  second-largest singular value $\sigma$. Consider the following random walk on $D_{2n}$: \begin{itemize}
        \item For $i$ odd, $\mu_i = \mu$.
        \item For $i$ even, $\mu_i(s) = p$ and $\mu_i(e) = 1 - p$.
    \end{itemize}
    In other words, to take $2k$ steps of this random walk, we take a sequence of $k$ random rotations, then flip $k$ weighted coins to decide whether to insert a reflection between each pair of random rotations.
    
    Say $i$ is even; we compute the singular value $\sigma_i$. In matrix form the operator $*(Q_*\mu_i)$ on $L^2(\Z/2\Z) \cong \R^2$ is given by $\begin{pmatrix}
        1 - p & p \\
        p & 1 - p
    \end{pmatrix}$. It is symmetric, so the singular values are just the absolute values of the eigenvalues. The vector $\begin{pmatrix}
        1 \\ -1
    \end{pmatrix}$ is a $(1 - 2p)$-eigenvector. Thus, the singular values of $*\mu_i$ are 1 and $|1 - 2p|$, so $\sigma_i = |1 - 2p|$. Theorem~\ref{thm:strong-random-walks} says that \[
    ||\mu_1 * \dots * \mu_{2k} - \pi||_{L^2}^2 \leq \sigma^k + |1 - 2p|^k.
    \]
    In particular, if $p = 1/2$, the random walk mixes on $D_{2n}$ half as fast as the $\mu$-random walk mixes on $\langle r\rangle \cong \Z/n\Z$.
\end{example}

Although Theorem~\ref{thm:strong-random-walks} makes weaker assumptions on the subgroups than Theorem~\ref{thm:intro-random-walks}, it is not possible to fully remove the normality assumption, as the following example shows:

\begin{example}\label{ex:normality-necessary}
    Consider the alternating group $A_5$. Recall that $A_5$ is generated by the 3-cycles $(1\; 2\; 3), (1\; 2\; 4), (1\; 2\; 5)$. 
    
    Consider the following three-step time-inhomogeneous ``random walk'' on $A_5$: $X_1$ is uniformly distributed on $\langle(1\; 2\; 3)\rangle$, $X_2$ is uniformly distributed on $\langle(1\; 2\; 4)\rangle$, and $X_3$ is uniformly distributed on $\langle(1\;2\;5)\rangle$. We wish to understand the random element $X_1X_2X_3 \in A_5$.

    The step distributions $\mu_1, \mu_2, \mu_3$ on the respective cyclic groups all have second-largest singular value zero. \details{The uniform probability measure sends every signed measure of total mass 0 to 0. Thus, its operator norm on $\mathcal{M}_0$ is 0.} So, if the conclusion of Theorem~\ref{thm:strong-random-walks} held in this situation, one might expect that $X_1X_2X_3$ is uniformly distributed on $A_5$. However, this is not the case. For instance, if $X_1X_2X_3$ were uniform on $A_5$, then $X_1X_2X_3$ would map $3$ to $4$ with probability $1/5$. But in fact, $X_1X_2X_3$ maps $3$ to $4$ with probability zero (because $X_2X_3$ fixes $3$ with probability 1, while $X_1$ can only map $3$ to $1$, $2$, or $3$ with positive probability).

    The essential obstruction to mixing in this case is evidently related to noncommutativity. The image of $3$ under $X_1X_2X_3$ is restricted because the only one of $X_1$, $X_2$, and $X_3$ that can move $3$ happens at the end of the sequence. What is really happening here is that even though the distribution of $X_1$ is uniform on each coset of $\langle(1\; 2\; 3)\rangle$, multiplication by $X_2$ destroys this property. We need some normality assumption on the subgroups involved so that taking additional random walk steps does not erase progress we have already made toward the uniform distribution.
\end{example}

The proof of Theorem~\ref{thm:strong-random-walks} mainly consists of linear algebra. When a group $G$ is finite, we can view measures on $G$ as finite tuples of numbers, i.e., as vectors in $\R^G$. We start by defining some notation describing this picture.
 
Consider the space $\mathcal{M} = L^2(G)$ of $\R$-valued functions on a finite group $G$. Since $G$ is finite, $\mathcal{M} \cong \R^G$ (with the Euclidean norm). Let $\mathcal{M}_0 = \{\nu \in \mathcal{M} \mid \nu(G) = 0\}$. Let $\mathcal{P} \subseteq \mathcal{M}$ be the set of signed measures $\nu$ on $G$ with $\nu(G) = 1$. Note that $\mathcal{P}$ and $\mathcal{M}$ are parallel affine hyperplanes in $\R^G$. Probability measures on $G$ are points in the simplex formed by the part of $\mathcal{P}$ in the positive orthant. The orthogonal complement to $\mathcal{M}_0$ is the line spanned by the uniform probability measure $\pi$ on $G$, and $\operatorname{span}\{\pi\}$ intersects $\mathcal{P}$ at $\pi$ and nowhere else.

Any measure $\mu_i$ on $G$ acts by on $\mathcal{M}$ by convolution on the right. If $\mu_i$ is a probability measure, the convolution operator $M_i\colon \nu \mapsto \nu*\mu_i$ also sends $\mathcal{P}$ into itself. The following lemma tells us that, in this case, $M_i$ contracts the distance between points of $\mathcal{P}$ and $\pi$.

\begin{lemma}\label{lem:convolution-contraction}
Let $G$ be any finite group and $\mu$ a probability measure on $G$. Let $M$ be the convolution operator $\nu \mapsto \nu * \mu$ and let $\sigma$ be the second-largest singular value of $M$ on $L^2(G)$.
\begin{enumerate}[label=(\arabic*)]
    \item If $\nu$ is a signed measure on $G$, then \[
    ||M\nu||_{L^2} \leq ||\nu||_{L^2}.
    \]
    \item If $\nu, \nu'$ are signed measures on $G$ with $\nu(G) = \nu'(G)$, then \[
    ||M\nu - M\nu'||_{L^2} \leq \sigma ||\nu - \nu'||_{L^2}.
    \]
\end{enumerate}
\end{lemma}
\begin{proof}
Part (1) is a case of Young's convolution inequality for unimodular groups\details{A \textit{unimodular group} is a group with a Haar measure that is both left and right invariant. To define the $L^2$ norm on $G$, we used the measure that assigns a value of 1 to all group elements, which is both left and right invariant. The proof of this general form is almost identical to the usual proof of Young's inequality via H\"older's inequality, and can be found on \href{https://en.wikipedia.org/wiki/Young\%27s\_convolution\_inequality\#Proof\_by\_H\%C3\%B6lder's\_inequality}{Wikipedia}}. To prove part (2), we will show that $\sigma$ is the $L^2$ operator norm of $M$ on the subspace of $L^2(G)$ consisting of signed measures with total mass 0.

Let $M^*$ be the adjoint operator to $M$. Observe that $M^*$ is also a convolution operator, given by $\nu \mapsto \nu * \check{\mu}$, where $\check{\mu}(g) = \mu(g^{-1})$ for $g \in G$. \details{
Suppose that $\delta_g$ is the indicator function which equals 1 on $g$ and 0 elsewhere. Then $\{\delta_g | g \in G\}$ gives a basis for $L^2(G)$. The $(g, h)$ entry of $M$ is $M\delta_h$ evaluated at $g$, which is \[
(\delta_h * \mu)(g) = \sum_{k \in H} \delta_h(k)\mu(k^{-1}g) = \mu(h^{-1}g).
\]
On the other hand, the $(h, g)$ entry of $M$ is $\mu(g^{-1}h) = \check{\mu}(h^{-1}g)$, which is the $(g, h)$ entry of the convolution operator corresponding to $\check{\mu}$. Thus, the convolution operator corresponding to $\check{\mu}$ is transpose to $M$, and since these operators have real entries, they are adjoint.
}\intuition{
If $X \sim \mu$, then the inner product $\ip{\delta_g}{M\delta_h}$ is $\PP[hX = g] = \PP[gX^{-1} = h] = \ip{\delta_h}{M^*\delta_g}$. 
}Thus, $M^*M$ is the convolution operator given by $\nu \mapsto \nu * \mu * \check{\mu}$. In particular, $M$ and $M^*M$ are each given by convolution with a probability measure, so they have a shared 1-eigenvector: the uniform measure $\pi$ on $G$. The largest singular value of a real matrix coincides with its $L^2$ operator norm. \details{
Let $e_1, \dots, e_d$ be an eigenbasis for $M^*M$ with eigenvalues $\sigma_1^2 \geq \dots \geq \sigma_d^2$.  Let $v = \sum_i a_i e_i$. Then \[
||Mv||_{L^2}^2 = \ip{Mv}{Mv} = \ip{v}{M^*Mv} = \ip{\sum_i a_ie_i}{\sum_i a_i\sigma_ie_i} = \sum_i \sigma_i^2 a_i^2.
\]
Then $||Mv||_{L^2}/||v||_{L^2}$ is maximized when $v = e_1$, and the maximum is $\sigma_1$.
}By part (1), the operator norm of $M$ is at most 1, so the largest eigenvalue of $M^*M$ is exactly 1. Let $L^2(G)_0 = \operatorname{span}\{\pi\}^\perp$. Since $\pi$ is an eigenvector of $M$, the operator $M$ restricts to an operator on $L^2(G)_0$, and since $(M|_{L^2(G)_0})^*(M|_{L^2(G)_0}) = (M^*M)|_{L^2(G)_0}$, the singular values of $M|_{L^2(G)_0}$ are the singular values of $M$ with a copy of 1 (the largest singular value of $M$) excluded. Thus, the operator norm and largest singular value of $M|_{L^2(G)_0}$ is the second-largest singular value of $M$, which is $\sigma$. If $\nu(G) = \nu'(G)$, then $\nu - \nu' \in L^2(H)_0$, so \[
||M(\nu - \nu')||_{L^2} \leq \sigma ||\nu - \nu'||_{L^2}.
\]
\end{proof} 
\begin{remark}
When the support of $\mu$ does not generate $G$, the conclusion of Lemma~\ref{lem:convolution-contraction}(2) still holds. However, we have $\sigma = 1$, so Lemma~\ref{lem:convolution-contraction}(2)   gives no useful information. \details{Suppose $\langle \operatorname{supp} \mu\rangle = H$. The operator $M^*M$ is convolution with $\mu * \check{\mu}$, which is still supported in $H$. Thus, $M^*M$ fixes the subspace $\mathcal{M}_H \subseteq L^2(G)$ of measures uniform on every left coset of $H$, which has dimension $[G : H]$ by Lemma~\ref{lem:pushforward-projection}. This means the multiplicity of 1 as a singular value of $M$ is at least $[G : H]$. If $H \subsetneq G$, then the second-largest singular value of $M$ is 1.}For this reason, in Theorem~\ref{thm:strong-random-walks}, we write $I_j = \{i \mid \langle \operatorname{supp}(\tilde{Q}_{j-1})_*\mu_i \rangle = H_j\}$ rather than $I_j = \{i \mid \operatorname{supp}(\tilde{Q}_{j-1})_*\mu_i \subseteq H_j\}$.
\end{remark}

The second part of Lemma~\ref{lem:convolution-contraction} implies that if the support of $\mu_i$ generates $G$ and $\nu$ is a probability measure, then $||M_i\nu - \pi||_{L^2} \leq \sigma_i||\nu - \pi||_{L^2}$, so applying $M_i$ to a probability measure $n$ times contracts the distance to $\pi$ by a factor of $\sigma_i^n$. More generally, if the support of each $\mu$ generates $G$, then applying any combination of $M_i$ in any order contracts the distance to $\pi$ by the appropriate product of $\sigma_i$ factors. However, when the support of $\mu_i$ is contained in a proper subgroup of $G$, the second-largest singular value of $M_i$ as an operator on $\mathcal{M}$ is always 1. In this case, Lemma~\ref{lem:convolution-contraction}(2) applied to $G$ gives no useful information.

The key idea in the proof of Theorem~\ref{thm:strong-random-walks} is that even though we cannot say outright that applying the operator $M_i$ moves a probability measure closer to uniform, we will show in Lemma~\ref{lem:subspace-contraction} that $M_i$ moves a probability measure closer to some (explicit) \textit{subspace} of $\mathcal{M}$ containing $\pi$. This subspace depends on the subgroup generated by the support of $\mu_i$. If we choose enough subspaces that their intersection is just $\{\pi\}$, then we will be able to show that successive application of different $M_i$'s moves a probability measure closer to that intersection, that is, to the uniform probability measure. The condition that our chosen subspaces intersect in $\{\pi\}$ is exactly the condition $G = \langle \bigcup_{H \in S} H\rangle$ in Theorem~\ref{thm:intro-random-walks}, or the condition that $G_k = \{e\}$ in Theorem~\ref{thm:strong-random-walks}.

For each subgroup $H \leq G$, let $\mathcal{M}_H \subseteq \mathcal{M}$ be the space of functions on $\mathcal{M}$ which are uniform on each left coset of $H$ (i.e., for $\nu \in \mathcal{M}_H$ and $g_1, g_2 \in G$ with $g_1^{-1}g_2 \in H$, $\nu(g_1) = \nu(g_2)$). As $H$ ranges over enough subgroups of $G$, these subspaces $\mathcal{M}_H$ are precisely the subspaces we want the $M_i$'s to move a measure closer to.

We now give two easy lemmas which will help us work with each subspace $\mathcal{M}_H$ concretely.

 \begin{lemma}\label{lem:projection}
    Let $G$ be a finite group and $H \leq G$ be a subgroup. Let $\nu \in \mathcal{M}$. Let $\Tilde{\nu} \in \mathcal{M}_H$ be the signed measure on $G$ given by $\Tilde{\nu}(gh) = \frac{\nu(gH)}{|H|}$ for $h \in H$. Then $\tilde{\nu}$ is the orthogonal projection of $\nu$ onto the subspace $\mathcal{M}_H$ of $\mathcal{M}$. In particular, $d_{L^2}(\nu, \mathcal{M}_H) = ||\nu - \Tilde{\nu}||_{L^2}$.
\end{lemma}
\begin{proof}
    We have decompositions \[
    \mathcal{M}_H = \bigoplus_{gH \in G/H} \operatorname{span}\{\pi_{gH}\} \subset  \bigoplus_{gH \in G/H} L^2(gH) = \mathcal{M},
    \]
    where $\pi_{gH} \in L^2(gH)$ is given by $\pi(gh) = 1/\sqrt{|H|}$ for all $h \in H$. The projection operator $\mathcal{M} \to \mathcal{M}_H$ decomposes as a direct sum of projection operators, one for each coset of $H$. In $L^2(gH)$, projection onto $\operatorname{span}\{\pi_{gH}\}$ is given by inner product with $\pi_{gH}$, and we have $\ip{\nu|_{gH}}{\pi_{gH}}\pi_{gH} = \tilde{\nu}|_{gH}$, which means the projection of $\nu$ onto $\mathcal{M}_H$ is $\tilde{\nu}$.
\end{proof}

\begin{lemma}\label{lem:normal-convolution}
Let $G$ be a finite group and $H$ a normal subgroup of $G$. Let $\mu \in L^2(G)$. Then $*\mu\colon L^2(G) \to L^2(G)$ preserves $\mathcal{M}_H$, i.e., $\mathcal{M}_H * \mu \subseteq \mathcal{M}_H$.
\end{lemma}
\begin{proof}
    Suppose $\nu \in \mathcal{M}_H$, so $\nu$ is uniform on each left coset of $H$. Since $H$ is normal, its left and right cosets coincide, so $\nu$ is uniform on each right coset of $H$. Say $X \sim \nu$ and $Y \sim \mu$ are independent, so $\nu * \mu$ is the distribution of $XY$. We have $\PP[X = hg] = \PP[X = h'g]$ for all $h, h' \in H$ and $g \in G$. For $y \in G$, we therefore have \[
    \PP[XY = hg \mid Y = y] = \PP[X = hgy^{-1}] = \PP[X = h'gy^{-1}] = \PP[XY = h'g \mid Y = y]
    \]
    for all $h, h' \in H$ and $g \in G$. Summing over $y$ shows $(\nu * \mu)(hg) = (\nu * \mu)(h'g)$ for all $h, h' \in H$ and $g \in G$. So, $\nu * \mu$ is uniform on right cosets (hence also left cosets) of $H$ and $\nu * \mu \in \mathcal{M}_H$.
\end{proof}

The previous two lemmas will be used together to show that convolution with a probability measure brings another probability measure closer to each subspace $\mathcal{M}_H \subseteq \mathcal{M}$ (corresponding to a normal subgroup $H \trianglelefteq G$). Lemma~\ref{lem:projection} allows us to relate the distance between a measure $\nu$ and the subspace $\mathcal{M}_H$ to the distance between $\nu$ and another measure $\widetilde{\nu}$; we will then use Lemma~\ref{lem:convolution-contraction} to show that this distance is contracted by convolving with a probability measure $\mu$. Then Lemma~\ref{lem:normal-convolution} is used to show that the resulting contracted distance bounds the difference between the new convolved probability measure $\nu * \mu$ and the same subspace $\mathcal{M}_H$. As a consequence, we obtain a bound on the distance between a convolution of many probability measures and the subspace $\mathcal{M}_H$:

\begin{lemma}\label{lem:subspace-contraction}
Let $G$ be a finite group and $H$ a normal subgroup of $G$. Let $\mu_1, \dots, \mu_n$ be probability measures on $G$ and $\nu_n = \mu_1 * \dots * \mu_n$. Let $I_H = \{1 \leq i \leq n \mid \langle\operatorname{supp} \mu_i\rangle = H\}$. For each $i \in I_H$, let $\sigma_i$ be the second-largest singular value of $*\mu_i$ on $H$. Let $\mathcal{M}_H \subseteq L^2(G)$ be the set of signed measures on $G$ that are uniform on left cosets of $H$. Then \[
d_{L^2}(\nu_n, \mathcal{M}_H)^2 \leq \frac{|G| - 1}{|G|}\prod_{i \in I_H} \sigma_i^2.
\] 
\end{lemma}
\begin{proof}
We say $\nu_0 = \delta_e$ is the Dirac measure on the identity of $G$, so that $\nu_{m+1} = \nu_m * \mu_{m+1}$ for all $0 \leq m < n$.

We will show the following statement by induction on $n$: \[\tag{$P(n)$}
d_{L^2}(\nu_n, \mathcal{M}_H)^2 \leq \frac{|G| - 1}{|G|}\prod_{\substack{i \in I_H \\ i \leq n}} \sigma_i^2.
\] 
First, we have $d_{L^2}(\delta_e, \mathcal{M}_H)^2 \leq ||\delta_e - \pi||_{L^2}^2 = 1 - \frac{1}{|G|}$. This proves $P(0)$. \details{Because $|\delta_e - \pi|$ takes the value $1/|G|$ at $|G| - 1$ elements and $1 - 1/|G|$ at one element, the squared $L^2$ norm is \[
||\delta_e - \pi||_{L^2}^2 = \frac{|G| - 1}{|G|^2} + \left(1 - \frac{1}{|G|}\right)^2 = \frac{1}{|G|} - \frac{1}{|G|^2} + 1 - \frac{2}{|G|} + \frac{1}{|G|^2} = 1 - \frac{1}{|G|}
\]}

Now suppose $P(n)$ holds. Let $\Tilde{\nu}_n$ be the orthogonal projection of $\nu_n$ onto $\mathcal{M}_H$, described explicitly in Lemma~\ref{lem:projection}. Then note that $||\nu_n - \Tilde{\nu}_n||_{L^2} = d_{L^2}(\nu_n, \mathcal{M}_H)$. 

There are two cases, depending on whether $n + 1 \in I_H$:

Suppose $\langle \operatorname{supp} \mu_{n+1} \rangle \neq H$, so $n + 1 \notin I_H$. By Lemma~\ref{lem:convolution-contraction}(1), $||\nu_{n+1} - \Tilde{\nu}_n*\mu_{n+1}||_{L^2} \leq ||\nu_n - \Tilde{\nu}_n||_{L^2}$. 

By Lemma~\ref{lem:normal-convolution}, we have $\Tilde{\nu}_n * \mu_{n+1} \in \mathcal{M}_H$, and \[
d_{L^2}(\nu_{n+1}, \mathcal{M}_H) \leq ||\nu_{n+1} - \Tilde{\nu}_n*\mu_{n+1}||_{L^2} \leq ||\nu_n - \Tilde{\nu}_n||_{L^2} = d_{L^2}(\nu_n, \mathcal{M}_H).
\]
Now suppose $\langle \operatorname{supp} \mu_{n+1} \rangle = H$, so $n + 1 \in I_H$. 

For $g \in G$, define $g_*\colon L^2(G) \to L^2(G)$ by $g_*\nu(h) = \nu(g^{-1}h)$. Note that $g_*$ is an automorphism of normed spaces, and for signed measures $\nu$ and $\mu$, we have $g_*(\nu * \mu) = g_*\nu * \mu$. \details{This is true by change of variables $k' = gk$:\begin{align*}
    g_*(\nu * \mu)(h) &= (\nu * \mu)(g^{-1}h) = \sum_{k\in G} \nu(k)\mu(k^{-1}g^{-1}h)  \\
    &= \sum_{k' \in G}\nu(g^{-1}k')\mu((k')^{-1}h) = \sum_{k' \in G}g_*\nu(k')\mu((k')^{-1}h) = (g_*\nu * \mu)(h).
\end{align*}}For each left coset $gH$ of $H$ we have \begin{align*}
||(\nu_n * \mu_{n+1})|_{gH} - (\Tilde{\nu}_n * \mu_{n+1})|_{gH}||_{L^2(gH)} &= ||(g^{-1}_*(\nu_n * \mu_{n+1}))|_{H} - (g^{-1}_*(\Tilde{\nu}_n * \mu_{n+1}))|_{H}||_{L^2(H)} \\
&= ||(g^{-1}_*\nu_n * \mu_{n+1})|_{H} - (g^{-1}_*\Tilde{\nu}_n * \mu_{n+1})|_{H}||_{L^2(H)}
\end{align*}
Since $\mu_{n+1}$ is supported on a subset of $H$, for any $\nu \in L^2(G)$ we have $(\nu * \mu_{n+1})|_{H} = \nu|_{H} * \mu_{n+1}|_H$. \details{
For $h \in H$ we have \[
(\nu * \mu_{n+1})(h) = \sum_{k \in G}\nu(k)\mu_{n+1}(k^{-1}h).
\]
The term $\mu_{n+1}(k^{-1}h)$ vanishes unless $k^{-1}h \in H$, so $k \in H_j$. Thus \[
(\nu * \mu_{n+1})(h) = \sum_{k \in H_j}\nu(k)\mu_{n+1}(k^{-1}h),
\]
which is the expression for convolution of $\nu|_{H}$ and $\mu|_{H}$ as signed measures on $H$.
}Thus, \[
||(\nu_n * \mu_{n+1})|_{gH} - (\Tilde{\nu}_n * \mu_{n+1})|_{gH}||_{L^2(gH)} = ||g^{-1}_*\nu_n|_{H} * \mu_{n+1}|_{H} - g^{-1}_*\Tilde{\nu}_n|_{H} * \mu_{n+1}|_{H_j}||_{L^2(H)}.
\]
By Lemma~\ref{lem:projection}, $\Tilde{\nu}_n|_{gH}$ is uniform on $gH$ with total mass $\nu_n(gH)$, so $g^{-1}_*\Tilde{\nu}_n|_{H}$ is uniform on $H$ with total mass $\nu_n(gH)$. Thus, $g^{-1}_*\Tilde{\nu}_n|_{H} * \mu_{n+1}|_{H} = g^{-1}_*\Tilde{\nu}_n|_{H}$.

Applying Lemma~\ref{lem:convolution-contraction}(2) on $H$, we get \begin{align*}
||g^{-1}_*\nu_n|_{H} * \mu_{n+1}|_{H} - g^{-1}_*\Tilde{\nu}_n|_{H}||_{L^2(H)} &\leq \sigma_{n+1}||g^{-1}_*\nu_n|_{H} - g^{-1}_*\Tilde{\nu}_n|_{H}||_{L^2(H)} \\
&= \sigma_{n+1}||(\nu_n - \Tilde{\nu}_n)|_{gH}||_{L^2(gH)}.
\end{align*}
\details{
We are using that pushforward and restriction are both linear maps.
}

Adding up over cosets of $H$ gives \[
||\nu_{n+1} - \Tilde{\nu}_n||_{L^2(G)} \leq \sigma_{n+1}||\nu_n - \Tilde{\nu}_n||_{L^2(G)}.
\]
Hence, \[
d_{L^2}(\nu_{n+1}, \mathcal{M}_{H}) \leq ||\nu_{n+1} - \Tilde{\nu}_n||_{L^2(G)} \leq \sigma_{n+1}||\nu_n - \Tilde{\nu}_n||_{L^2(G)} = \sigma_{n+1}d_{L^2}(\nu_n, \mathcal{M}_{H}).
\]
By induction, we get \[
d_{L^2}(\nu_n, \mathcal{M}_{H})^2 \leq \frac{|G| - 1}{|G|}\prod_{\substack{i \in I_{H} \\ i \leq n}} \sigma_i^2.
\]
\end{proof}

Note that if $H$ is a subgroup of $G$ and $P\colon G \twoheadrightarrow G/H$ is the projection onto the left coset space, then one can identify $L^2(G/H)$ with $\mathcal{M}_H$ as follows:

\begin{lemma}\label{lem:pushforward-projection}
Let $G$ be a finite group and $H$ a subgroup. Let $\mathcal{M}_H \subset L^2(G)$ be the space of measures uniform on left cosets of $H$. Let $P\colon G \twoheadrightarrow G/H$ send each element to the corresponding left coset of $H$. Then the map $\phi\colon L^2(G) \to L^2(G/H)$ sending $\mu$ to $|H|^{-1/2} P_*\nu$ restricts to an isometry of normed spaces $\phi|_{\mathcal{M}_H}\colon \mathcal{M}_H \cong L^2(G/H)$. Moreover, $(\phi|_{\mathcal{M}_H})^{-1}\circ \phi$ is the orthogonal projection map $L^2(G) \twoheadrightarrow \mathcal{M}_H$.
\end{lemma}
\begin{proof}
The map $\phi|_{\mathcal{M}_H}$ is norm-preserving because if $\nu$ is uniform on left cosets of $H$, then $\nu(gH) = |H|\nu(g)$ for $g \in G$. Indeed, we have \[
|||H|^{-1/2} P_*\nu||_{L^2(G/H)} = \frac{1}{|H|}\sum_{gH \in G/H} |\nu(gH)|^2 = \frac{1}{|H|^2} \sum_{g \in G} |\nu(gH)|^2 = \frac{1}{|H|^2}\sum_{g \in G} |H|^2 |\nu(g)|^2 = ||\nu||_{L^2(G)}.
\]
The inverse map $(\phi|_{\mathcal{M}_H})^{-1}$ is given by $(\phi|_{\mathcal{M}_H})^{-1}(\mu)(g) = \mu(gH)/|H|$. The fact that $(\phi|_{\mathcal{M}_H})^{-1}\circ \phi$ is the orthogonal projection map $L^2(G) \to \mathcal{M}_H$ follows from Lemma~\ref{lem:projection}.
\end{proof}

Identifying $L^2(G/H)$ with $\mathcal{M}_H$ will allow us to prove Theorem~\ref{thm:strong-random-walks} by induction. Using Lemma~\ref{lem:subspace-contraction}, we can say that a random walk approaches the subspace $\mathcal{M}_H$. Then, we can consider its projection onto $\mathcal{M}_H$ as a random walk on $G/H$. This allows us to ignore all random walk steps supported in $H$. The key ingredient that allows us to combine Lemma~\ref{lem:subspace-contraction} with the inductive hypothesis is the following lemma:

\begin{lemma}\label{lem:rw-l2-distance}
Let $G$ be a finite group and $H \leq G$. Let $\pi$ be the uniform distribution on $G$ and let $\mu$ be any signed measure on $G$. Let $P\colon G \twoheadrightarrow G/H$ be the set map sending each element of $G$ to the corresponding left coset of $H$. Let $\mathcal{M}_H \subseteq L^2(G)$ be the set of signed measures uniform on left cosets of $H$. Then \[
||\mu - \pi||_{L^2(G)}^2  = \frac{1}{|H|}||P_* \mu - P_* \pi||_{L^2(G/H)}^2 + d_{L^2}(\mu, \mathcal{M}_H)^2
\]
\end{lemma}
\begin{proof}
Let $\tilde{\mu}$ be the orthogonal projection of $\mu$ onto $\mathcal{M}_H$. By Lemma~\ref{lem:projection}, we have $P_*\mu = P_*\tilde{\mu}$. Then \begin{align*}
||\mu - \pi||_{L^2(G)}^2 &= ||\tilde{\mu} - \pi||_{L^2(G)}^2 + ||\mu - \tilde{\mu}||_{L^2(G)}^2 \\
&= \frac{1}{|H|}||P_* \tilde{\mu} - P_* \pi||_{L^2(G/H)}^2 + d_{L^2}(\mu, \mathcal{M}_H)^2 \\
&= \frac{1}{|H|}||P_* \mu - P_* \pi||_{L^2(G/H)}^2 + d_{L^2}(\mu, \mathcal{M}_H)^2
\end{align*}
\end{proof}

The final lemma before the proof of Theorem~\ref{thm:strong-random-walks} is a pair of facts about pushforwards to quotients:
\begin{lemma}\label{lem:pushforward-facts}
Let $G$ be a finite group and $H$ a normal subgroup. Let $P\colon G \twoheadrightarrow G/H$ be the projection. Then:
\begin{enumerate}[label=(\arabic*)]
    \item If $\mu, \nu \in L^2(G)$, we have $P_*(\mu * \nu) = P_*\mu * P_*\nu$.
    \item Suppose $\mu \in L^2(G)$ and the second-largest singular value of $*\mu$ on $\langle \operatorname{supp} \mu\rangle$ is $\sigma$. Then the second-largest singular value of $*(P_*\mu)$ on $P(\langle \operatorname{supp} \mu\rangle)$ is at most $\sigma$.
\end{enumerate}
\end{lemma}
\begin{proof}
\begin{enumerate}[label=(\arabic*)]
    \item We have \[
    P_*(\mu * \nu)(gH) = \sum_{h \in H} (\mu * \nu)(gh) = \sum_{h \in H}\sum_{k \in G} \mu(k)\nu(k^{-1}gh) = \sum_{kH \in G/H}\sum_{h \in H} \mu(kh)\nu(h^{-1}k^{-1}gH)
    \]
    \details{
    \begin{align*}
    P_*(\mu * \nu)(gH) &= \sum_{h \in H} (\mu * \nu)(gh) = \sum_{h \in H}\sum_{k \in G} \mu(k)\nu(k^{-1}gh) = \sum_{k \in G}\sum_{h \in H}\mu(k)\nu(k^{-1}gh) \\
    &= \sum_{k \in G} \mu(k) \nu(k^{-1}gH) = \sum_{kH \in G/H}\sum_{h \in H} \mu(kh)\nu(h^{-1}k^{-1}gH)
    \end{align*}
    }Since left cosets of $H$ are right cosets too, $h^{-1}k^{-1}gH = k^{-1}gH$, so \[
    P_*(\mu * \nu)(gH) = \sum_{kH \in G/H}\sum_{h \in H} \mu(kh)\nu(k^{-1}gH) = \sum_{kH \in G/H} \mu(kH)\nu(k^{-1}gH) = P_*\mu * P_*\nu.
    \]
    \item By restricting to the projection $\langle \operatorname{supp} \mu\rangle \twoheadrightarrow P(\langle \operatorname{supp} \mu\rangle)$ we may as well assume $\langle \operatorname{supp} \mu \rangle = G$.
    
    Recall (from the proof of Lemma~\ref{lem:convolution-contraction}(2)) that the second-largest singular value of $*\mu$ is the operator norm of $*\mu$ acting on the subspace $L^2(G)_0$ of measures with total mass 0. Suppose $\sigma'$ is the second-largest singular value of $*(P_*\mu)$. Then \[
    \sigma' = \sup_{\substack{\nu \in L^2(G/H) \\ \nu(G/H) = 0}} \frac{||\nu * P_*\mu||_{L^2(G/H)}}{||\nu||_{L^2(G/H)}}
    \]
    Let $\mathcal{M}_H \subseteq L^2(G)$ be the space of signed measures uniform on cosets of $H$. By Lemma~\ref{lem:pushforward-projection} and part (1) of this lemma, we have \[
    \sup_{\substack{\nu \in L^2(G/H) \\ \nu(G/H) = 0}} \frac{||\nu * P_*\mu||_{L^2(G/H)}}{||\nu||_{L^2(G/H)}} = \sup_{\substack{\tilde{\nu} \in \mathcal{M}_H \\ \tilde{\nu}(G) = 0}} \frac{||P_*\tilde{\nu} * P_*\mu||_{L^2(G/H)}}{||\tilde{\nu}||_{L^2(G/H)}} = \sup_{\substack{\tilde{\nu} \in \mathcal{M}_H \\ \tilde{\nu}(G) = 0}} \frac{||P_*(\tilde{\nu} * \mu)||_{L^2(G/H)}}{||P_*\tilde{\nu}||_{L^2(G/H)}}
    \]
    Then by Lemma~\ref{lem:pushforward-projection} again, we have \[
    \sigma' = \sup_{\substack{\tilde{\nu} \in \mathcal{M}_H \\ \tilde{\nu}(G) = 0}} \frac{||P_*(\tilde{\nu} * \mu)||_{L^2(G/H)}}{||P_*\tilde{\nu}||_{L^2(G/H)}} = \sup_{\substack{\tilde{\nu} \in \mathcal{M}_H \\ \tilde{\nu}(G) = 0}} \frac{||\tilde{\nu} * \mu||_{L^2(G)}}{||\tilde{\nu}||_{L^2(G)}} \leq \sup_{\substack{\tilde{\nu} \in L^2(G) \\ \tilde{\nu}(G) = 0}} \frac{||\tilde{\nu} * \mu||_{L^2(G)}}{||\tilde{\nu}||_{L^2(G)}} = \sigma
    \]
\end{enumerate}
\end{proof}

\begin{proof}[Proof of Theorem~\ref{thm:strong-random-walks}]
We will prove the following statement by induction on $r$: \begin{align*}\label{eq:random-walk-inductive-step}
||(\Tilde{Q}_r)_* \nu_n - (\Tilde{Q}_r)_* \pi||_{L^2(G_r)}^2 \leq \sum_{j = r + 1}^k \frac{|G_{j-1}| - 1}{|G_r|}\left(\prod_{i \in I_j} \sigma_i^2\right).\tag{$P(r)$}
\end{align*}
When $r = k$, the right hand side of $P(r)$ is 0. Since both $(\Tilde{Q}_r)_* \nu_n$ and $(\Tilde{Q}_r)_* \pi$ are the unique probability measure on $G_r = \{e\}$, the left hand side is also 0, so $P(k)$ holds.

Now suppose $P(r + 1)$ holds. We will show $P(r)$ holds.

Since $(\Tilde{Q}_r)_* \pi$ is the uniform distribution on $G_r$, Lemma~\ref{lem:rw-l2-distance} applied to $G_r$ and $H_{r+1}$ says \begin{align*}
||(\Tilde{Q}_r)_* \nu_n - (\Tilde{Q}_r)_* \pi||_{L^2(G_r)}^2 &= \frac{1}{|H_{r+1}|}||(\Tilde{Q}_{r+1})_* \nu_n - (\Tilde{Q}_{r+1})_* \pi||_{L^2(G_{r+1})}^2 + d_{L^2}((\Tilde{Q}_r)_* \nu_n, \mathcal{M}_{H_{r+1}})^2
\end{align*}
where $\mathcal{M}_{H_{r+1}}$ is the subspace of $L^2(G_r)$ consisting of measures uniform on cosets of $H_{r+1}$. By the inductive hypothesis, \begin{align*}
\frac{1}{|H_{r+1}|}||(\Tilde{Q}_{r+1})_* \nu_n - (\Tilde{Q}_{r+1})_* \pi||_{L^2(G_{r+1})}^2 &\leq \sum_{j = r + 2}^k \frac{|G_{j-1}| - 1}{|G_{r+1}||H_{r+1}|}\left(\prod_{i \in I_j} \sigma_i^2\right) \\
&= \sum_{j = r + 2}^k \frac{|G_{j-1}| - 1}{|G_r|}\left(\prod_{i \in I_j} \sigma_i^2\right).
\end{align*}
By Lemma~\ref{lem:pushforward-facts}(1), we have $(\Tilde{Q}_r)_* \nu_n = (\Tilde{Q}_r)_* \mu_1 * \dots * (\Tilde{Q}_r)_* \mu_n$, so by Lemma~\ref{lem:subspace-contraction} applied to $G_r$, $H_{r+1}$, and the measures $(\Tilde{Q}_r)_* \mu_i$, we get \[
d_{L^2}((\Tilde{Q}_r)_* \nu_n, \mathcal{M}_{H_{r+1}})^2 \leq \frac{|G_r| - 1}{|G_r|} \prod_{i \in I_{r + 1}} \sigma_i^2.
\]
Hence, \begin{align*}
    ||(\Tilde{Q}_r)_* \nu_n - (\Tilde{Q}_r)_* \pi||_{L^2(G_r)}^2 & \leq \sum_{j = r + 2}^k \frac{|G_{j-1}| - 1}{|G_r|}\left(\prod_{i \in I_j} \sigma_i^2\right) + \frac{|G_r| - 1}{|G_r|}\prod_{i \in I_{r+1}} \sigma_i^2 \\
    &= \sum_{j = r + 1}^k \frac{|G_{j-1}| - 1}{|G_r|}\left(\prod_{i \in I_j} \sigma_i^2\right),
\end{align*}
completing the induction. When $r = 0$, we get \[
||\nu_n - \pi||_{L^2}^2 \leq \sum_{j=1}^k \frac{|G_{j-1}| - 1}{|G|}\left(\prod_{i \in I_j} \sigma_i^2\right).
\]
\end{proof}

Now we show how Theorem~\ref{thm:strong-random-walks} implies Theorem~\ref{thm:intro-random-walks} by giving a corollary slightly stronger than Theorem~\ref{thm:intro-random-walks}.

\begin{corollary}\label{cor:random-walks}
Let $G$ be a finite group, and let $\mu_1, \mu_2, \dots, \mu_n$ be probability measures on $G$. For each subgroup $H$ of $G$, let $I_H = \{i \mid H = \langle \operatorname{supp} \mu_i\rangle\}$. Let $H_1, \dots, H_k$ be a finite set of subgroups of $G$ such that $G = \left\langle \bigcup_{j = 1}^k H_j\right\rangle$ and the image of $H_j$ in $G/H_1\cdots H_{j-1}$ is a normal subgroup for all $1 \leq j \leq k$. Write $\nu_n = \mu_1 * \dots * \mu_n$. Also, for each $i$, let $\sigma_i$ be the second-largest singular value of $*\mu_i$ as an operator on $L^2(\langle \operatorname{supp} \mu_i \rangle)$. Let $\pi$ be the uniform distribution on $G$. 

If $I_{H_j}$ is nonempty for each $1 \leq j \leq k$, we have \[
||\nu_n - \pi||_{L^2} \leq \sum_{j = 1}^k \left(\prod_{i \in I_{H_j}} \sigma_i\right).
\]
\end{corollary}
\begin{proof}
We have $\bigcup_{j = 1}^k H_j \subseteq H_1 \cdots H_k$, so $H_1\cdots H_k = G$. Let $\tilde{Q}_j\colon G \to G/H_1\cdots H_j$ be the projection. For $i \in I_{H_j}$, let $\sigma_i'$ be the second-largest singular value of $*(\tilde{Q}_{j-1})_*\mu_i$ on $\tilde{Q}_{j-1}(\langle \operatorname{supp} \mu_i\rangle) = \tilde{Q}_{j-1}(H_j)$. By Lemma~\ref{lem:pushforward-facts}(2), we have $\sigma_i' \leq \sigma_i$. Then applying Theorem~\ref{thm:strong-random-walks} to the sequence \[
G \longrightarrow G/H_1 \longrightarrow G/H_1H_2 \longrightarrow \dots \longrightarrow G/H_1\cdots H_k = \{e\}
\]
we get that \[
||\nu_n - \pi||_{L^2}^2 \leq \sum_{j=1}^k \frac{|G/H_1\cdots H_{j-1}| - 1}{|G|} \left(\prod_{i \in I_{H_j}} (\sigma_i')^2 \right) \leq \sum_{j=1}^k \left(\prod_{i \in I_{H_j}} \sigma_i^2 \right).
\]
Then the corollary follows by subadditivity of square root.
\end{proof}
We obtain Theorem~\ref{thm:intro-random-walks} from Corollary~\ref{cor:random-walks} because the image of a normal subgroup under a surjection is normal.

\section{Universality for Random Groups}\label{sect:universality}

The goal of this section is to prove Theorem~\ref{thm:intro-matrix-universality}.

To prove Theorem~\ref{thm:intro-matrix-universality}, we will use the moment method of Wood (see \cite{wood2017sandpile, wood2019matrices}) as follows. Let $a > 0$ be an integer and $A$ the set of isomorphism classes of finite abelian groups with exponent dividing $a$. Let $X_1, X_2, \dots$ be a sequence of random finitely generated abelian groups and $Y$ be a random abelian group with $Y \otimes \Z/a\Z$ finite. If for every $G \in A$ we have \[
\lim_{n\to\infty} \E[\#\operatorname{Sur}(X_n, G)] = \E[\#\operatorname{Sur}(Y, G)] \leq |\wedge^2 G| \tag{$*$}
\]
then for every $H \in A$ we have \[
\lim_{n\to\infty} \PP[X_n \otimes \Z/a\Z \cong H] = \PP[Y \otimes \Z/a\Z \cong H] \tag{$**$}
\]
(this follows from \cite[Theorem 3.1]{wood2019matrices} by taking $Y_n \coloneqq Y \otimes \Z/a\Z$ a finite group). The quantity $\E[\#\operatorname{Sur}(X_n, G)]$ is called the \textit{$G$-moment} of $X_n$.

\begin{remark}\label{rmk:weak-convergence}
We say an abelian group $G$ is \textit{well-behaved} if $G \otimes \Z/a\Z$ is finite for every positive integer $a$. The class of well-behaved abelian groups includes all finitely generated abelian groups. Moreover, $\lambda_u$ is supported on well-behaved abelian groups even for $u = 0$.

We can put a topology on the set of (isomorphism classes of) well-behaved abelian groups\footnote{Note the difference from the published version, which is not quite correct.} given by a basis of open sets of the form \[
U_{a, H} = \{X \text{ well-behaved abelian} \mid X \otimes \Z/a\Z \cong H\}\]
indexed by positive integers $a$ and finite abelian groups $H$ of exponent dividing $a$. The assertion that $(**)$ holds for all choices of $a$ and $H$ is equivalent to the assertion that the distribution of $X_n$ converges weakly to the distribution of $Y$ in this topology. In particular, if $(*)$ holds for all well-behaved abelian groups $G$, then the distribution of $X_n$ converges weakly to the distribution of $Y$. See \cite{liu2018freegroup} for more details on this topology in a slightly different setting.

We explain this in more detail in the proof of Theorem~\ref{thm:intro-matrix-universality} in Section~\ref{sect:weak-convergence} at the end of this paper.
\end{remark}
If $Y \sim \lambda_u$, then \cite[Lemma 3.2]{wood2019matrices} gives\[
\E[\#\operatorname{Sur}(Y, G)] = |G|^{-u}.
\]
Following this strategy, we obtain Theorem~\ref{thm:intro-matrix-universality} as a corollary of Theorem~\ref{thm:moments}, which states that if $X_n$ are the cokernels of $n\times(n+u)$ random matrices satisfying appropriate conditions, then $\lim_n\E[\#\operatorname{Sur}(X_n, G)] = |G|^{-u}$.

When $X_n$ is the cokernel of a random $n\times m$ matrix $M$, the problem of counting surjections from $X_n$ into $G$ can be attacked with combinatorics. Say $X_n = \Z^n/\Lambda$, where $\Lambda$ is a random subgroup of $\Z^n$ (e.g., the column space of a random integer matrix). Then surjections $X_n \to G$ correspond one-to-one with surjections $\Z^n \to G$ which vanish on $\Lambda$. It follows from linearity of expectation that \[
\E[\#\operatorname{Sur}(\Z^n/\Lambda, G)] = \sum_{f \in \operatorname{Sur}(\Z^n, G)} \PP[f(\Lambda) = 0].
\]
\details{
\[
\E[\#\operatorname{Sur}(\Z^n/\Lambda, G)] = \E\left[\sum_{f \in \operatorname{Sur}(\Z^n, G)} 1_{f(\Lambda) = 0}\right] = \sum_{f \in \operatorname{Sur}(\Z^n, G)}\E[1_{f(\Lambda) = 0}] =  \sum_{f \in \operatorname{Sur}(\Z^n, G)} \PP[f(\Lambda) = 0].
\]
}In the case of cokernels of random matrices, $\Lambda$ is the subgroup generated by the columns of the random matrix, viewed as random elements of $\Z^n$. But we can also view $M$ as a random element of $(\Z^n)^m$. Given a map $f\colon \Z^n \to G$, we get by abuse of notation a map $f\colon (\Z^n)^m \to G^m$ applying $f$ to each component. Then we have that $f(\Lambda) = 0$ if and only if $f(M) = 0$. Thus, we want to bound the probabilities $f(M) = 0$. Past work on random matrices with independent entries (e.g., \cite{nguyenRandomIntegralMatrices2022}) has observed that if $Z$ is a random tuple in $\Z^n$ with independent, sufficiently regular components, then for most $f \in \operatorname{Sur}(\Z^n, G)$, the element $f(Z) \in G$ is close to uniformly distributed. Applying this to each column independently allows us to compute $\PP[f(M) = 0]$. In this work, we apply the same principle to consider several columns of a random matrix at a time. 

\subsection{Balanced elements}

The following definition captures the idea that a random element in a group is not too concentrated in a particular coset.

\begin{definition}\label{def:balanced}
    Let $G$ be a group. A $G$-valued random variable $X$ is \textit{$\varepsilon$-balanced} if for any proper subgroup $H < G$ and element $g \in G$, we have $\PP[X \in gH] \leq 1 - \varepsilon$. 
\end{definition}
This definition agrees with the definition in \cite{wood2019matrices} when $G$ is a finite cyclic group. Here is an example of an $\varepsilon$-balanced random variable that does not take values in a cyclic group.

\begin{example}
    Let $G$ be a finitely generated group with finite generating set $S$ containing the identity, and let $X$ be a random variable supported on $S$ with $\min_{g \in S} \PP[X = g] = \varepsilon$. Then $X$ is $\varepsilon$-balanced.

    Indeed, suppose $H$ is a subgroup of $G$ and $g \in G$ such that $\PP[X \in gH] > 1 - \varepsilon$. Then $\PP[X \in gH] = 1$, so $S \subset gH$. Since $S$ contains the identity element of $G$, we must have $gH = H$, and since $S \subset gH = H$, we must have $H = G$.
\end{example}

In this paper, we consider $n \times m$ integer matrices as elements of the abelian group $(\Z^n)^m$. For each subset $S$ of $[n] \times [m]$, we have a quotient map $\pi_S$ from $(\Z^n)^m$ onto $\Z^S$ given by taking the entries of a matrix indexed by pairs in $S$. We say that a subset of the entries of a random matrix $M$ with indices $S$ is \textit{jointly $\varepsilon$-balanced} if $\pi_S(M)$ is $\varepsilon$-balanced in $\Z^S$.

The new definition of $\varepsilon$-balanced has some desirable properties that help construct new examples of $\varepsilon$-balanced random variables.

\begin{lemma}\label{lem:balanced-facts}
    \begin{enumerate}[label=(\arabic*)]
        \item If $\pi\colon G \twoheadrightarrow Q$ is a surjective homomorphism of groups and $X$ is $\varepsilon$-balanced in $G$, then $\pi(X)$ is $\varepsilon$-balanced in $Q$. 
        \item Let $G, G'$ be groups, $X$ be $\varepsilon$-balanced in $G$, and $Y$ be $\varepsilon$-balanced in $G'$. If $X$ and $Y$ are independent, then $(X, Y)$ is $\varepsilon$-balanced in $G\times G'$.
    \end{enumerate}
\end{lemma}
\begin{proof}
\begin{enumerate}[label=(\arabic*)]
    \item Let $qK \subsetneq Q$ be a coset of a proper subgroup of $Q$. Let $\tilde{q} \in \pi^{-1}(q)$, so $\pi^{-1}(qK) = \tilde{q}\pi^{-1}(K)$ is a coset of a proper subgroup of $G$. \details{We have $\pi^{-1}(qK) = \bigcup_{g\in \pi^{-1}(q)}g\pi^{-1}(K)$. Since $\pi$ is surjective, $\pi^{-1}(K)$ is a proper subgroup of $G$, or else $K = \pi(\pi^{-1}(K)) = \pi(G) = Q$. Now for any $g, g' \in \pi^{-1}(q)$, we have $\pi(g) = \pi(g')$, so $gg^{-1} \in \ker \pi \subset \pi^{-1}(K)$ and $g\pi^{-1}(K) = g'\pi^{-1}(K)$. Hence, if $g \in \pi^{-1}(q)$, we have $\pi^{-1}(qK) = g\pi^{-1}(K)$, and this is a coset of a proper subgroup of $G$.}Since $X$ is $\varepsilon$-balanced, \[
    \PP[\pi(X) \in qK] \leq \PP[X \in \tilde{q}\pi^{-1}(K)] \leq 1 - \varepsilon,
    \]
    as desired.
    
    \item Let $kH$ be a coset of a proper subgroup of $G \times G'$. Note that \[
    \PP[(X, Y) \in kH] = \PP[(X, e) \in (e, Y^{-1})kH] = \PP[(X, e) \in (e, Y^{-1})kH \cap (G \times \{e\})].
    \]
    Recall that the intersection of two cosets in a group is either empty or a coset of their intersection. In particular, $(e, Y^{-1})kH \cap (G \times \{e\})$ is either empty or a coset of a subgroup of $G \times \{e\}$.
    
    There are two cases, depending on whether $(e, Y^{-1})kH \cap (G \times \{e\})$ is always a proper subset of $G \times \{e\}$: \begin{enumerate}[label=\roman*.]
        \item If $(e, y^{-1})kH \cap (G \times \{e\}) \subsetneq G \times \{e\}$ for all $y \in G'$:

        Condition on $Y = y$ for some fixed $y \in G'$. Since $X$ and $Y$ are independent, and $X$ is $\varepsilon$-balanced, \[
        \PP[(X, e) \in (e, Y^{-1})kH \cap (G \times \{e\}) \mid Y = y] = \PP[(X, e) \in (e, y^{-1})kH \cap (G \times \{e\})] \leq 1- \varepsilon.
        \]
        \details{Since $(e, y^{-1})kH \cap (G \times \{e\}) \subsetneq G \times \{e\}$, either $(e, y^{-1})kH \cap (G \times \{e\}) = \varnothing$ or $(e, y^{-1})kH \cap (G \times \{e\})$ is a coset of a proper subgroup of $G \times \{e\}$. In the former case, $\PP[(X, e) \in (e, y^{-1})kH \cap (G \times \{e\})] = 0$. In the latter case, notice that $(X, e)$ is $\varepsilon$-balanced in $G \times \{e\}$ by (1). Hence $\PP[(X, e) \in (e, y^{-1})kH \cap (G \times \{e\})] \leq 1 - \varepsilon$. 

        In both cases, $\PP[(X, y) \in kH] = \PP[(X, e) \in (e, y^{-1})kH \cap (G \times \{e\})] \leq 1 - \varepsilon$. Hence, we have \begin{align*}
        \PP[(X, Y) \in kH] &= \sum_{y\in G'}\PP[(X, e) \in (e, Y^{-1})kH \cap (G \times \{e\}) \mid Y = y] \PP[Y = y] \\
        &\leq (1 - \varepsilon)\sum_{y\in G'}\PP[Y = y] \\
        &= 1 - \varepsilon.
        \end{align*}
        }

        \item If $G \times \{e\} \subseteq (e, y^{-1})kH$ for some $y \in Y$, then $(e, e) \in (e, y^{-1})kH$, so in particular $(e, y^{-1})kH$ is a subgroup of $G \times G'$ and we must have $(e, y^{-1})kH = H$. We claim that $H = G \times H'$ for some proper subgroup $H'$ of $G'$.

        Indeed, let $\pi\colon G\times G' \to G'$ be the projection and let $H' = \pi(H)$. On one hand, clearly $H \subseteq G \times \pi(H)$. On the other, if $(g, h') \in G \times H'$, then $h' = \pi(g', h)$ for some $(g', h) \in H$. Then $(g, h') = (g(g')^{-1}, e)(g', h)$. Since $(g(g')^{-1}, e) \in G \times \{e\} \subseteq H$, we have $(g, h') \in H$. Hence $H = G \times H'$. Note that $H' \lneq G'$, or else $H = G \times G'$ is not a proper subgroup.

        Then \[
        \PP[(X, Y) \in kH] = \PP[Y \in H'] \leq 1 - \varepsilon.
        \]
    \end{enumerate}
    Hence, in both cases we have $\PP[(X, Y) \in kH] \leq 1 - \varepsilon$ and since this holds for every proper coset $kH$, we have that $(X, Y)$ is balanced.
\end{enumerate}  
\end{proof}
Note that Lemma~\ref{lem:balanced-facts} gives us a nice way to build up $\varepsilon$-balanced matrices. If the entries of a random matrix can be partitioned into independent subsets and each of these subsets of the entries is jointly $\varepsilon$-balanced, then the whole matrix is $\varepsilon$-balanced. For example, any matrix with independent, $\varepsilon$-balanced entries (as in \cite{wood2019matrices}) is $\varepsilon$-balanced as a matrix. 

When a random variable is $\varepsilon$-balanced, we can get an upper bound on the associated singular value.
\begin{lemma}\label{lem:balanced-singular-values}
    Suppose $G$ is a finite group and $X$ is $\varepsilon$-balanced in $G$ with distribution $\mu$. Let $\sigma$ be the second largest singular value of the operator $* \mu$ on $L^2(G)$. Then \[
    \sigma \leq \exp\left(-\frac{\varepsilon}{2|G|^3}\right).
    \]
\end{lemma}
\begin{proof}
    Note that $\sigma$ is the square root of the second largest eigenvalue of the operator $*\nu\coloneqq * \mu * \check{\mu}\colon L^2(G) \to L^2(G)$, where $*\check{\mu}$ is the adjoint to the operator $*\mu$, given by $\check{\mu}(g) = \mu(g^{-1})$. The operator $*\nu$ is the transition operator for a random walk on $G$, where each step is a difference of two independent copies of $X$. 

    In particular, note that $\nu = \check{\nu}$. \details{
    If $g \in G$, then \[
    \nu(g) = \sum_{h \in G} \mu(h)\mu(g^{-1}h) = \sum_{h \in G} \mu(gh)\mu(h) = \nu(g^{-1}).
    \]
    }For any generating set $\Sigma$ of $G$, \cite[Theorem 6.2]{saloff-coste2004book} applied to $\Sigma \cup \Sigma^{-1}$ shows that the second-largest eigenvalue $\sigma^2$ of $*\nu$ is bounded above by \[
    \sigma^2 \leq 1 - \frac{m}{D^2},
    \]
    where $m = \min_{x\in \Sigma}\nu(x)$ and $D$ is the diameter of the Cayley graph of $(G, \Sigma)$. In particular, $D \leq |G|$. \details{
    As stated, \cite[Theorem 6.2]{saloff-coste2004book} applies to \textit{symmetric} generating sets $\Sigma$. However, if $\Sigma$ is any generating set, we can let $\Sigma' = \Sigma \cup \Sigma^{-1}$. The diameter of the resulting Cayley graph is at most $D$ and the minimal mass of any element of $\Sigma'$ is the same as that of $\Sigma$. Thus, we still get the inequality above.
    }

    The goal is to choose an appropriate $\Sigma$ to bound $m$ from below. Note that if $X_1$ and $X_2$ are $\varepsilon$-balanced and independent, then so is $X_1X_2^{-1}$ (via conditioning on $X_2$). In particular, $\nu$ is $\varepsilon$-balanced.

    We proceed iteratively. Having chosen $x_1, \dots, x_{n-1}$ (including the empty set $n = 1$), if $\langle x_1, \dots, x_{n-1} \rangle = G$ then we are done. Otherwise, since $\nu$ is $\varepsilon$-balanced, $\nu(\langle x_1, \dots, x_{n-1} \rangle) \leq 1 - \varepsilon$. Choose \[
    x_n = \operatorname{argmax}_{x\in G\setminus \langle x_1, \dots, x_{n-1} \rangle} \nu(x).
    \] 
    Since $\nu(\langle x_1, \dots, x_{n-1} \rangle) \leq 1 - \varepsilon$, we have $\nu(G\setminus \langle x_1, \dots, x_{n-1} \rangle) \geq \varepsilon$, so $\nu(x_n) \geq \frac{\varepsilon}{|G\setminus \langle x_1, \dots, x_{n-1} \rangle|} \geq \frac{\varepsilon}{|G|}$.

    Hence we have $m \geq \frac{\varepsilon}{|G|}$, so \[
    \sigma \leq \sqrt{1 - \frac{\varepsilon}{|G|^3}} \leq 1 - \frac{\varepsilon}{2|G|^3} \leq \exp\left(-\frac{\varepsilon}{2|G|^3}\right),
    \]
    as desired.
\end{proof}

With some more work, we can get a better bound on the singular values when $G$ is abelian using an argument based on \cite[Lemma 2.1, Lemma 2.2]{wood2017sandpile}:

\begin{lemma}\label{lem:balanced-singular-values-abelian}
    Suppose $G$ is a finite abelian group of exponent dividing $a$ and $X$ is $\varepsilon$-balanced in $G$ with distribution $\mu$. Let $\sigma$ be the second largest singular value of the operator $*\mu$ on $L^2(G)$. Then \[
    \sigma \leq \exp\left(-\frac{\varepsilon}{a^2}\right).
    \]
\end{lemma}
\begin{proof}
As in the proof of Lemma~\ref{lem:balanced-singular-values}, we consider the operator $*\nu\coloneqq *\mu*\check{\mu}$. The second-largest eigenvalue $\sigma^2$ of $*\nu$ is equal to the spectral radius of the restriction $(*\nu)|_{L^2(G)_0}$, where $L^2(G)_0$ is the space of signed measures on $G$ of total mass $0$, as in Section~\ref{sect:walks}.

The spectral radius of $(*\nu)|_{L^2(G)_0}$ is bounded above by the operator norm $||(*\nu|_{L^2(G)_0})^n||^{1/n}$ for any natural number $n$, i.e., \[
\sigma^2 \leq ||(*\nu|_{L^2(G)_0})^n||^{1/n}.
\] 
We will compute this upper bound.

Let $\lambda \in L^2(G)_0$ with $||\lambda||_{L^2(G)_0} = 1$, and let $Y \sim \lambda + \pi$ be a $G$-valued random variable, where $\pi$ is the uniform distribution on $G$. Let $X_1, \dots, X_n, \check{X}_1, \dots, \check{X}_n$ be independent, identically distributed $G$-valued random variables drawn from $\mu$. Then $Y + \sum_{i=1}^nX_i - \sum_{i=1}^n \check{X}_i$ is distributed according to $\lambda*(\nu)^{*n} + \pi$, so we have \[
||\lambda*(\nu)^{*n}||_{L^2(G)_0}^2 = \sum_{g \in G} \left|\PP\left[Y + \sum_{i=1}^nX_i - \sum_{i=1}^n \check{X}_i = g\right] - \frac{1}{|G|}\right|^2.
\]

Let $G^* = \operatorname{Hom}(G, S^1)$ be the Pontryagin dual of $G$.

By the discrete Fourier transform, we have \begin{align*}
    \left|\PP\left[Y + \sum_{i=1}^nX_i - \sum_{i=1}^n \check{X}_i = g\right] - \frac{1}{|G|}\right| &= \left|\frac{1}{|G|}\sum_{\chi \in G^* \setminus \{0\}}\E\left[\chi\left(Y + \sum_{i=1}^nX_i - \sum_{i=1}^n \check{X}_i - g\right)\right]\right| \\
    &\leq \frac{1}{|G|}\sum_{\chi \in G^* \setminus \{0\}}\left|\E\left[\chi(Y - g)\right]\right|\prod_{i=1}^n \left|\E\left[\chi(X_i)\right]\E\left[-\chi(\check{X}_i)\right] \right|
\end{align*}

We examine the terms $\E\left[\chi(X_i)\right]$ and $\E\left[-\chi(\check{X}_i)\right]$ for $\chi$ a character of $G$. By Lemma~\ref{lem:balanced-facts}(a), the random variable $\chi(X_i)$ is $\varepsilon$-balanced in the (necessarily cyclic) subgroup $\chi(G) \subset S^1$. By \cite[Lemma 2.2]{wood2017sandpile}, since $\chi(X_i)$ is a $\varepsilon$-balanced random element of a nontrivial subgroup of the $a$th roots of unity, we have $|\E[\chi(X_i)]| \leq \exp\left(-\frac{\varepsilon}{a^2}\right)$. We have the same bound for $|\E\left[-\chi(\check{X}_i)\right]|$.

So, \begin{align*}
||\lambda*(\nu)^{*n}||_{L^2(G)_0}^2 &\leq \sum_{g \in G}\left(\frac{1}{|G|}\sum_{\chi \in G^*\setminus\{0\}}\left|\E\left[\chi(Y - g)\right]\right|\prod_{i=1}^n \left|\E\left[\chi(X_i)\right]\E\left[-\chi(\check{X}_i)\right] \right| \right)^2 \\
&\leq |G|\exp\left(-\frac{4n\varepsilon}{a^2}\right)
\end{align*}
which means \[
||(*\nu|_{L^2(G)_0})^n||^{1/n} \leq |G|^{1/2n}\exp\left(-\frac{2\varepsilon}{a^2}\right)
\]
and, taking the limit as $n\to\infty$, we obtain \[
\sigma \leq \exp\left(-\frac{\varepsilon}{a^2}\right)
\]
as we wanted. 
\end{proof}

Now we will use the $\varepsilon$-balanced condition to give a related balancedness condition for matrices that contains information about how balanced and independent the entries are. 

\begin{definition}
    Let $S$ be a finite set. A \textit{partition} of $S$ is a collection $\mathcal{P} = \{P_1, \dots, P_k\} \subseteq 2^S$, such that $S = P_1 \sqcup P_2 \sqcup \dots \sqcup P_k$ and each $P_i$ is nonempty. We say $|\mathcal{P}| = \max_i \#P_i$ and $\#\mathcal{P} = k$.  If $\sigma\subseteq 2^{S}$, write $\cup\sigma$ for $\bigcup_{S\in\sigma} S$.
\end{definition}
Note that $\#\mathcal{P}\cdot |\mathcal{P}| \geq \#S$.

The next definition specifies the kinds of restrictions we will give for the matrices in our universality class. The idea is that we can split up the columns of the matrix and then the rows, so that the resulting sections of the matrix are $\varepsilon$-balanced.

If $M$ is an $n\times m$ matrix, $S = \{s_1 < \dots < s_k\} \subset [n]$, and $T = \{t_1 < \dots < t_\ell\} \subset [m]$, then $M_{S, T}$ is the $k\times \ell$ matrix $(M_{s_i, t_j})_{1 \leq i \leq k, 1 \leq j \leq \ell}$.  
\begin{definition}\label{def:w-h-e-balanced}
    An $n\times m$ random matrix $M$ with entries in a ring $R$ is $(w, h, \varepsilon)$-balanced if there is a partition $\mathcal{Q} = \{Q_1, \dots, Q_r\}$ of $[m]$ and a partition $\mathcal{P} = \{P_1, \dots, P_\ell\}$ of $[n]$ with $|\mathcal{Q}| \leq w$, $|\mathcal{P}| \leq h$, and such that each random matrix $M_{P_i, Q_j}$ is $\varepsilon$-balanced in the additive abelian group $(R^{\#P_i})^{\#Q_j}$ and the random matrices $M_{P_i, Q_j}$ are independent.
\end{definition}
If $|\mathcal{P}| = |\mathcal{Q}| = 1$ then we recover the definition of $\varepsilon$-balanced from \cite{wood2019matrices} and other related work. 

Now we are ready to state the main theorem of this section:
\begin{theorem}\label{thm:matrix-universality}
    Let $u \geq 0$ be an integer. Let $(w_n)_n, (h_n)_n$, $(\varepsilon_n)_n$ be sequences of real numbers such that $w_n = O(n^{\alpha_1})$, $h_n = O(n^{\alpha_2})$ and $\varepsilon_n = \Omega( n^{-\beta})$ for some $0 \leq \alpha_1, \alpha_2, \beta < 1$ satisfying \[2\alpha_1 + \alpha_2 < 1 - 2\beta.
    \]
    For each integer $n \geq 0$, let $M_n$ be an $(w_n, h_n, \varepsilon_n)$-balanced $n \times (n + u)$ random matrix with entries in $\Z$. Let $Y \sim \lambda_u$. Then for all positive integers $a$ and abelian groups $H$ of exponent dividing $a$ we have \[
    \lim_{n\to\infty} \PP[\operatorname{coker}(M_n) \otimes \Z/a\Z \cong H] = \PP[Y \otimes \Z/a\Z \cong H] = \lambda_u(U_{a, H}).
    \]
\end{theorem}
Together with Remark~\ref{rmk:weak-convergence}, this gives Theorem~\ref{thm:intro-matrix-universality}. A more detailed proof is given in Section~\ref{sect:weak-convergence}.

As discussed at the beginning of this section, we will prove this by computing the limiting moments of $\operatorname{coker}(M_n)$, which involves estimating $\PP[f(M_n) = 0]$ for maps $\Z^n \to G$.
\begin{remark}
The same proof will work as written when the entries of $M_n$ come from any ring $R$ with at most one quotient to $\Z/a\Z$ for any positive integer $a$. Some examples of interest are the $p$-adic integers $\Z_p$ or a product $\prod_i \Z_{p_i}$ for some collection of distinct primes $p_i$. We will find that when $R$ has exactly one quotient to $\Z/a\Z$, then for any finite abelian group $G$ of exponent dividing $a$, the limiting $G$-moment of $\operatorname{coker}(M_n)$ is $|G|^{-u}$. Then we get the conclusion of Theorem~\ref{thm:matrix-universality} for those $a$ for which $R$ has a quotient to $\Z/a\Z$.
\end{remark}

\subsection{Bounds for most maps}

It turns out that $(w, h, \varepsilon)$-balanced is a strong enough condition that we can get bounds on $\PP[f(M) = 0]$ for the vast majority of maps $f$.

\begin{definition}
    If $V$ is an abelian group with generating set $S$ and $T \subseteq S$, we write $V_{\setminus T}$ for the subgroup $\langle S \setminus T \rangle$ of $V$. When $V = (\Z/a\Z)^n$ or $\Z^n$ we implicitly take $S$ to be the ``standard basis''.
    
    Let $\mathcal{P} = \{P_1, \dots, P_\ell\}$ be a partition of $S$ and $G$ be a finite abelian group. A function $f\colon V \to G$ is a \textit{$\mathcal{P}$-code} of distance $w$ if for any $\sigma \subset \mathcal{P}$ with $|\cup \sigma| < w$, we have $f(V_{\setminus\cup\sigma}) = G$.
\end{definition}

To approximate $\PP[f(M) = 0]$ for codes $f$, we will split the matrices $M$ into independent sets of columns. Each such set of $r$ random columns gets mapped to something close to uniform in $G^r$. The following lemma is analogous to \cite[Lemma 2.1]{wood2019matrices}.

\begin{lemma}\label{lem:partition-equidistribution}
    Let $n, r \geq 1$ be integers. Let $G$ be a finite abelian group and let $a$ be a multiple of the exponent of $G$. Let $N$ be the number of subgroups of $G$. Let $\varepsilon > 0$ be a real number. Let $V = (\Z/a\Z)^n$. Let $\mathcal{P} = \{P_i\}$ be a partition of $[n]$ and let $\ell = |\mathcal{P}|$. Let $f \in \operatorname{Hom}(V, G)$ be a $\mathcal{P}$-code of distance $w < n$. 
    
    Let $M$ be an $n\times r$ random matrix in $V^r$ such that the matrices $M_{P_i, [r]}$ are independent and $\varepsilon$-balanced as random elements of $((\Z/a\Z)^{\#P_i})^r$.
    
    Let $g_1, \dots, g_r \in G$. Then \[
    |\PP[f(M) = (g_1, \dots, g_r)] - |G|^{-r}| \leq N\exp\left(-\frac{\varepsilon w}{\ell Na^2}\right)
    \]
\end{lemma}
\begin{proof}
    Let $e_1, \dots, e_n$ be the standard generating set for $V$. For $i = 1, \dots, \#\mathcal{P}$, let $V_i = \langle e_j \mid j \in P_i\rangle \cong \Z^{\#P_i}$.

    The idea is to treat $f(M)$ as a random walk in $G^r$. We have \[
    f(M) = \sum_{i=1}^{\#\mathcal{P}}f(M_{P_i, [r]}),
    \]
    where $M_{P_i, [r]}$ is interpreted as an $\varepsilon$-balanced random element of $V_i^r \cong ((\Z/a\Z)^{\#P_i})^r$, a subgroup of $((\Z/a\Z)^n)^r$. 

    Let $S = \{H \leq G^r \mid H = f(V_i^r) \text{ for at least }w/\ell N\text{ values of }i\}$. Note that $f(V_i^r) = f(V_i)^r$, so as $i$ ranges over $1, \dots, \#\mathcal{P}$ there are at most $N$ possible values for $f(V_i^r)$, each an $r$th power of a subgroup of $G$. Let $I = \{i \mid f(V_i^r) \notin S\}$. Then $\#I \leq w/\ell$, and so $|\bigcup_{i \in I} P_i| \leq w$. Since $f$ is a $\mathcal{P}$-code of distance $w$, it remains surjective if we discard all of these indices, which means the images of the $V_i^r$s with $f(V_i^r) \in S$ generate $G^r$. In other words, we have $\langle\bigcup_{H \in S} H\rangle = G^r$. The subgroups in $S$ will be the ones we use in the random walk, applying Theorem~\ref{thm:intro-random-walks}.

    By the definition of $S$, for each $H$ in $S$ we have $\#I_H \geq w/\ell N$. By Lemma~\ref{lem:balanced-facts}, the steps $f(M_{P_i, [r]})$ are $\varepsilon$-balanced, which means that by Lemma~\ref{lem:balanced-singular-values-abelian} the second largest singular value $\sigma_i$ of the $i$th step $f(M_{P_i, [r]})$ is bounded above: $\sigma_i \leq \exp\left(-\frac{\varepsilon}{a^2}\right)$ (using the fact that each $f(M_{P_i, [r]})$ is supported on a subgroup of $G^r$, which still has exponent dividing $a$).
    
    Hence by Theorem~\ref{thm:intro-random-walks} we have \[
    |\PP[f(M) = (g_1, \dots, g_r)] - |G|^{-r}| \leq \sum_{H\in S}\exp\left(-\frac{\varepsilon w}{\ell Na^2}\right) \leq N\exp\left(-\frac{\varepsilon w}{\ell Na^2}\right),
    \]
    as desired.
\end{proof}

To combine these estimates we will use a result in the flavor of \cite[Lemma 2.3]{wood2019matrices}:

\begin{lemma}\label{lem:err-combining}
    Let $x_1, \dots, x_m \geq -1$ be real numbers such that $\sum_{i=1}^m \max\{0, x_i\} \leq \log 2$. Then \[
    \left|\prod_{i=1}^m (1 + x_i) - 1\right| \leq 2\sum_{i=1}^m |x_i|
    \]
    and \[
    \sum_{i=1}^m \min\{0, x_i\} \leq  \prod_{i=1}^m (1 + x_i) - 1 \leq 2\sum_{i=1}^m \max\{0, x_i\}.
    \]
\end{lemma}
\begin{proof}
    The first statement follows from the second statement because $\max\{0, x_i\} \leq |x_i|$ and $\min\{0, x_i\} \geq -|x_i|$. So, we will show the second statement.

    First, assume $x_i \leq 0$ for all $i$. In that case, \[
    \prod_{i=1}^m (1 + x_i) \geq 1 + \sum_{i=1}^m x_i.
    \]
    Next, assume $x_i \geq 0$ for all $i$. Using the fact that $1 + x_i \leq e^{x_i}$, we get \[
    \prod_{i=1}^m (1 + x_i) \leq e^{\sum_{i=1}^m x_i}.
    \]
    We have $e^x - 1 = 2x$ at $x = 0$ and $\frac{d}{dx}(e^x - 1) \leq \frac{d}{dx}(2x)$ for $x \leq \log 2$, so $e^x - 1 \leq 2x$ for $0 \leq x \leq \log 2$. Hence, if $\sum_{i=1}^m x_i \leq \log 2$, then $\exp\left(\sum_{i=1}^m x_i\right) - 1 \leq 2\sum_{i=1}^m x_i$.

    Now consider the general case. By replacing each negative $x_i$ with zero, we can only increase the product $\prod_{i=1}^m (1 + x_i)$. On the other hand, by replacing each positive $x_i$ with zero, we can only decrease it. Hence, for general $x_i$, we get \[
    \sum_{i=1}^m \min\{0, x_i\} \leq \prod_{i=1}^m (1 + \min\{0, x_i\}) - 1 \leq \prod_{i=1}^m (1 + x_i) - 1 \leq \prod_{i=1}^m (1 + \max\{0, x_i\}) - 1 \leq 2\sum_{i=1}^m \max\{0, x_i\}.
    \]
\end{proof}

Applying this lemma with $x_i$ being the error in Lemma~\ref{lem:partition-equidistribution} multiplied by $|G|^r$ yields an estimate on the probability that the whole matrix maps to zero:

\begin{lemma}\label{lem:full-matrix-partition-equidistribution}
    Let $u \geq 0$ be an integer. Let $G$ be a finite abelian group and let $a$ be a multiple of the exponent of $G$. Let $(w_n)_n, (h_n)_n, (\delta_n)_n, (\varepsilon_n)_n$ be sequences of real numbers such that $h_n\log n = o\left(n\varepsilon_n\delta_n\right)$ and $w_nh_n = o\left(n\varepsilon_n\delta_n\right)$.
    
    For a natural number $n$, let $V = (\Z/a\Z)^n$. Let $M$ be an $(w_n, h_n, \varepsilon_n)$-balanced $n\times (n + u)$ random matrix with entries in $\Z/a\Z$. Let $\mathcal{P}$ be the row partition associated to $M$ and let $f \in \operatorname{Hom}(V, G)$ be a $\mathcal{P}$-code of distance $n\delta_n$. 

    Then for all $g_1, \dots, g_{n+u} \in G$, we have \[
    |\PP[f(M) = (g_1, \dots, g_{n+u})] - |G|^{-n-u}| = o\left(\frac{1}{|G|^{n+u}}\right).
    \]
\end{lemma}
\begin{proof}
    Let $\mathcal{P}$ and $\mathcal{Q}$ be the row and column partitions for $M$ as in the definition of $(w_n, h_n, \varepsilon_n)$-balanced. Let $M_i = M_{[n], Q_i}$ for each $i$. Let $g_{Q_i} = (g_j \mid j \in Q_i)$. By independence, \[
    \PP[f(M) = (g_1, \dots, g_{n+u})] = \prod_i \PP[f(M_i) = g_{Q_i}].
    \]
    For each $i$, let $x_i = |G|^{\#Q_i}\PP[f(M_i) = g_{Q_i}] - 1$. By Lemma~\ref{lem:partition-equidistribution}, we have \begin{align*}
    |x_i| &\leq N|G|^{\#Q_i}\exp\left(-\frac{n\varepsilon_n \delta_n}{Nh_na^2}\right) \\
    &\leq N|G|^{w_n}\exp\left(-\frac{n\varepsilon_n \delta_n}{Nh_na^2}\right).
    \end{align*}
    Hence we have \[
    \log|x_i| \leq \log N + w_n\log |G| - \frac{n\varepsilon_n \delta_n}{N h_na^2}.
    \]
    Since $h_n\log n = o(n\varepsilon_n\delta_n)$, we have $\lim_{n\to\infty} \frac{n\varepsilon_n\delta_n}{2Nh_na^2} - \log n = \infty$. In particular, $\log|x_i| \to -\infty$ as $n\to\infty$.
    
    By Lemma~\ref{lem:err-combining}, we therefore have that for such $n$, \begin{align*}
    ||G|^{n+u}\PP[f(M) = (g_1, \dots, g_{n+u})] - 1| &= \left|\prod_{i=1}^m|G|^{\#Q_i}\PP[f(M_i) = g_{Q_i}]| - 1\right| \\
    &= \left|\prod_{i=1}^m (1 + x_i) - 1\right| \\
    &\leq 2\sum_{i=1}^m |x_i| \\
    &\leq 2Nn\exp\left(-\frac{n\varepsilon_n\delta_n}{2Nh_na^2}\right).
    \end{align*}
    Since $\lim_{n\to\infty} \frac{n\varepsilon_n\delta_n}{2Nh_na^2} - \log n = \infty$, the right hand side converges to $0$ as $n\to\infty$.
\end{proof}

\subsection{Bounds for the rest of the maps}

This gives results for the case when $f$ is a code, but we still need to account for non-codes. To do this, we will show that non-codes make up a negligible proportion of all maps $V \to G$ and thus contribute only a small error term to the sum $\E[\#\operatorname{Sur}(\operatorname{coker}(M), G)]$. However it turns out that splitting maps into codes and non-codes is not enough to get this bound. Instead, as in \cite{wood2019matrices}, \cite{nguyenRandomIntegralMatrices2022}, and similar work, we will categorize non-codes by how far they are from being codes.

If $D$ is an integer with prime factorization $\prod_i p_i^{e_i}$, we write $\ell(D) = \sum_i e_i$.

\begin{definition}\label{def:depth}
    If $V = (\Z/a\Z)^n$ and $\mathcal{P}$ is a partition of the ``standard basis'' of $V$, the \textit{$(\mathcal{P},\delta)$-depth} of $f \in \operatorname{Hom}(V, G)$ is the maximal positive $D$ such that there is a $\sigma \subset \mathcal{P}$ with $|\cup \sigma| < \ell(D)\delta n$ such that $D = [G : f(V_{\setminus\cup\sigma})]$, or 1 if there is no such $D$.
\end{definition}
We can count the number of $f$ that have given $(\mathcal{P}, \delta)$-depth:
\begin{lemma}\label{lem:count-depth}
    If $D > 1$, then the number of $f \in \operatorname{Hom}(V, G)$ of $(\mathcal{P}, \delta)$-depth $D$ is at most \[
    K\binom{n}{\lceil \ell(D)\delta n\rceil - 1}2^{\ell(D)\delta n}|G|^nD^{-n + \ell(D)\delta n},
    \]
    where $K$ is the number of subgroups of $G$ of index $D$.
\end{lemma}
\begin{proof}
    For each $f$ of $(\mathcal{P},\delta)$-depth $D$, there is a $\sigma \subset \mathcal{P}$ as described in Definition~\ref{def:depth}. There must be some set $S \subset [n]$ with $\#S = \lceil \ell(D)\delta n\rceil - 1$ and $\cup\sigma \subseteq S$. There are $\binom{n}{\lceil \ell(D)\delta n\rceil - 1}$ choices of $S$, and for each choice of $S$, there are certainly at most $2^{\#S} = 2^{\lceil \ell(D)\delta n\rceil - 1} \leq 2^{\ell(D)\delta n}$ choices of $\cup\sigma$. Since $\mathcal{P}$ is a partition, $\cup\sigma$ uniquely determines $\sigma$, so there are at most $2^{\ell(D)\delta n}$ choices of $\sigma$ for each choice of $S$. 
    
    Now we count how many $f$ of $(\mathcal{P},\delta)$-depth $D$ have each choice of $\sigma$, so fix $\sigma$. There are $K$ subgroups of $G$ of index $D$, so there are $K$ options for $f(V_{\setminus\cup\sigma})$.

    Fix a subgroup $H$ of $G$ with index $D$. We now count the number of $f$ with $f(V_{\setminus\cup \sigma}) \subseteq H$. There are at most $|H|^{n - |\cup\sigma|}$ maps from $V_{\setminus \cup\sigma}$ to $H$, and for each such map, there are at most $|G|^{|\cup\sigma|}$ homomorphisms from $V$ to $G$ which restrict appropriately. Hence, there are at most \[
    |H|^{n - |\cup\sigma|}|G|^{|\cup\sigma|} = |G|^{n - |\cup\sigma|}D^{-n + |\cup\sigma|}|G|^{|\cup\sigma|} = |G|^nD^{-n + |\cup\sigma|} \leq |G|^nD^{-n + \ell(D)\delta n}
    \]
    maps $f$ with $f(V_{\setminus\cup\sigma}) = H$. Combined with the counts of choices of $\sigma$ and subgroups of $G$ of index $D$, we get the lemma.
\end{proof}

For non-codes, we do not get precise estimates on $\PP[f(M) = 0]$, but we can get upper bounds.
\begin{lemma}\label{lem:partition-depth}
    Let $r \geq 1$ be an integer. Let $G$ be a finite abelian group and let $a$ be a multiple of the exponent of $G$. Let $N$ be the number of subgroups of $G$. Let $\varepsilon > 0$ and $\delta > 0$ be real numbers. Let $V = (\Z/a\Z)^n$. Let $\mathcal{P} = \{P_1, \dots, P_m\}$ be a partition of $[n]$ and let $\ell = |\mathcal{P}|$. Let $f \in \operatorname{Hom}(V, G)$ have $(\mathcal{P}, \delta)$-depth $D > 1$ with $[G : f(V)] < D$. 
    
    Let $M$ be an $n\times r$ random matrix in $V^r$ such that the matrices $M_{P_i, [r]}$ are independent and $\varepsilon$-balanced as random elements of $((\Z/a\Z)^{\#P_i})^r$.
    
    Then \[
    \PP[f(M) = 0]  \leq (1 - \varepsilon)\left(D^r|G|^{-r} + N\exp\left(-\frac{\varepsilon \delta n}{2N\ell a^2}\right)\right).
    \]
\end{lemma}
\begin{proof}
    Since $f$ has $(\mathcal{P},\delta)$-depth $D$, there is a $\sigma \subset \mathcal{P}$ with $|\cup\sigma| < \ell(D)\delta n$ such that $D = [G : f(V_{\setminus\cup\sigma})]$. Let $f(V_{\setminus\cup\sigma}) =\colon H$. Since $[G : f(V)] < D$, we cannot have that $\sigma$ is empty. 
    
    Write $f(M) = \sum_{j\notin\sigma}f(M_{P_j, [r]}) + \sum_{j\in\sigma}f(M_{P_j, [r]})$. So, \[
    \PP[f(M) = 0] = \PP[f(M) \in H]\PP\left[\sum_{j\notin \sigma}f(M_{P_j, [r]}) = -\sum_{j\in\sigma}f(M_{P_j, [r]}) \ \middle|\  f(M) \in H\right].
    \]
    We bound the two probabilities on the right side separately. Note that since $\sum_{j\in\sigma}f(M_{P_j, [r]}) \in H$, we have $f(M) \in H$ exactly when $\sum_{j\notin\sigma}f(M_{P_j, [r]}) \in H$. Since $[G : f(V)] < [G : H]$, there must be some $i \in \sigma$ such that $f(M_{P_i, [r]})$ reduces to a nonzero element of $G/H$. Conditioning on all other $M_{P_k, [r]}$ for $k \neq i$, by the $\varepsilon$-balanced assumption we have that \[
    \PP\left[f(M) \in H\right] = \PP\left[f(M_{P_i, [r]}) \equiv -\sum_{j \in \sigma\setminus\{i\}}f(M_{P_j, [r]}) \pmod{H}\right] \leq 1 - \varepsilon.
    \]
    For the second probability, let $\mathcal{P}'$ be the partition of $[n]\setminus\cup\sigma$ induced by $\mathcal{P}$. Notice that $f|_{V_{\setminus\cup\sigma}}$ is a $\mathcal{P}'$-code of distance $\delta n$. Indeed, suppose there is some $\tau \subset \mathcal{P}'$ with $|\tau| < \delta n$ inducing some $\tau' \subset \mathcal{P}$ with $f(V_{\setminus\cup(\sigma\cup\tau)}) \neq H$. Then the image of $f|_{V_{\setminus\cup(\sigma\cup\tau)}}$ would have index strictly greater than $D$, contradicting maximality of $D$.
    
    Now we can apply Lemma~\ref{lem:partition-equidistribution} to the submatrix $M_{[n]\setminus\cup\sigma, [r]}$ and the code $f$ mapping it into $H^r$. If $N'$ is the number of subgroups of $H$ and $\ell' = |\mathcal{P}'|$, then conditioning on $M_{P_j, [r]}$ for $j \in \sigma$ gives \begin{align*}
    \PP\left[\sum_{j\notin \sigma}f(M_{P_j, [r]}) = -\sum_{j\in\sigma}f(M_{P_j, [r]}) \ \middle|\  f(M) \in H\right] &\leq |H|^{-r} + N'\exp\left(-\frac{\varepsilon \delta n}{2N'\ell'a^2}\right) \\
    &\leq D^r|G|^{-r} + N\exp\left(-\frac{\varepsilon \delta n}{2N\ell a^2}\right),
    \end{align*}
    and the lemma follows.
\end{proof}
Finally, we use Lemma~\ref{lem:err-combining} again to get a bound for the full $n\times (n+u)$ matrix:
\begin{lemma}\label{lem:full-partition-depth}
    Let $u \geq 0$ be an integer. Let $G$ be a finite abelian group and let $a$ be a multiple of the exponent of $G$. Let $(w_n)_n, (h_n)_n, (\delta_n)_n, (\varepsilon_n)_n$ be sequences of real numbers such that $h_n\log n = o\left(n\varepsilon_n\delta_n\right)$ and $w_nh_n = o\left(n\varepsilon_n\delta_n\right)$.
    
    For a natural number $n$, let $V = (\Z/a\Z)^n$. Let $M$ be an $(w_n, h_n, \varepsilon_n)$-balanced $n\times (n + u)$ random matrix with entries in $\Z/a\Z$. Let $\mathcal{P}$ be the row partition associated to $M$ and let $f \in \operatorname{Hom}(V, G)$ have $(\mathcal{P}, \delta_n)$-depth $D > 1$, with $[G : f(V)] < D$.

    Then there is a constant $K > 0$ depending only on $u$, $G$, $\alpha$, $\beta$, and the sequences $h_n$, $w_n$ such that for all $n$, \[
    \PP[f(M) = 0] \leq K\exp\left(-\varepsilon_n\frac{n}{w_n}\right)D^n|G|^{-n}.
    \]
\end{lemma}
\begin{proof}
    Let $\mathcal{Q}$ be the column partition for $M$ as in the definition of $(w_n, h_n, \varepsilon_n)$-balanced. Let $M_i = M_{[n], Q_i}$ for each $i$. By independence, \[
    \PP[f(M) = 0] = \prod_i \PP[f(M_i) = 0].
    \]
    For each $i$, let $x_i = \frac{|G|^{\#Q_i}D^{-\#Q_i}}{1 - \varepsilon_n}\PP[f(M_i) = 0] - 1$. By Lemma~\ref{lem:partition-depth}, we have \begin{align*}
    \max\{0, x_i\} &\leq N|G|^{\#Q_i}D^{-\#Q_i}\exp\left(-\frac{n\varepsilon_n \delta_n}{2Nh_n a^2}\right) \\
    &\leq N|G|^{w_n}D^{-w_n}\exp\left(-\frac{n\varepsilon_n \delta_n}{2Nh_n a^2}\right).
    \end{align*}
    By the same argument as in the proof of Lemma~\ref{lem:full-matrix-partition-equidistribution}, we find that \begin{align*}
    \frac{(D^{-1}|G|)^{n+u}}{(1 - \varepsilon_n)^{\#\mathcal{Q}}}\PP[f(M) = 0] - 1 = o(1)
    \end{align*}
    and \[
    \PP[f(M) = 0] = (1 + o(1))(1 - \varepsilon_n)^{\#\mathcal{Q}}(D^{-1}|G|)^{-n-u} \leq (1 + o(1))\exp(-\varepsilon_n\#\mathcal{Q})(D^{-1}|G|)^{-n-u}
    \]
    The conclusion follows from $\#\mathcal{Q} \geq \frac{n}{w_n}$.
\end{proof}

\subsection{Computing the moments}

Finally, we can combine all these results to compute the limiting moments for cokernels of $(w_n, h_n, \varepsilon_n)$-balanced random matrices. The most delicate part of this proof is the part where we handle the non-codes. This will involve a careful choice of the sequence $\delta_n$.
\begin{theorem}\label{thm:moments}
    Let $u \geq 0$ be an integer. Let $G$ be a finite abelian group and let $a$ be a multiple of the exponent of $G$ (including zero). Let $(w_n)_n, (h_n)_n$, $(\varepsilon_n)_n$ be sequences of real numbers such that $w_n = O(n^{\alpha_1})$, $h_n = O(n^{\alpha_2})$ and $\varepsilon_n = \Omega(n^{-\beta})$ for some real numbers $0 \leq \alpha_1, \alpha_2, \beta < 1$ satisfying \[2\alpha_1 + \alpha_2 < 1 - 2\beta.
    \]
    For each natural number $n$, let $M_n$ be an $(w_n, h_n, \varepsilon_n)$-balanced $n \times (n + u)$ random matrix with entries in $\Z/a\Z$. Then \[
    \lim_{n\to\infty}\E[\#\operatorname{Sur}(\operatorname{coker}(M_n), G)] = |G|^{-u}.
    \]
\end{theorem}
When $a = 0$, we mean that $M_n$ is a matrix over $\Z$. The conditions on $w_n, h_n, \varepsilon_n$ can be weakened somewhat with more careful accounting.
\begin{proof}
    Let $V = (\Z/a\Z)^n$. Following the discussion at the beginning of this section, we have \[
    \E[\#\operatorname{Sur}(\operatorname{coker}(M_n), G)] = \sum_{f \in \operatorname{Sur}(V, G)}\PP[f(M_n) = 0].
    \]\details{
    A surjection $\operatorname{coker}(M) \to G$ corresponds to a surjection $\Z^n \to G$ vanishing on the column space of $M$. Since $G$ has exponent $a$, any such surjection factors through $\Z^n/a\Z^n = (\Z/a\Z)^n$.
    }\counterwithout{equation}{section}Thus, \begin{align*}
        \left|\E[\#\operatorname{Sur}(\operatorname{coker}(M_n), G)] - \frac{1}{|G|^u}\right|
        &=\left|\sum_{f \in \operatorname{Sur}(V, G)}\PP[f(M_n) = 0] - \frac{1}{|G|^u}\right| \\
        &= \left|\sum_{f \in \operatorname{Sur}(V, G)}\PP[f(M_n) = 0] - \sum_{f \in \operatorname{Hom}(V, G)}\frac{1}{|G|^{n + u}}\right| \\
        &\leq \sum_{f \in \operatorname{Sur}(V, G)}\left|\PP[f(M_n) = 0] - \frac{1}{|G|^{n + u}}\right|.
    \end{align*}
    We will break this sum into a term coming from codes and a few terms coming from non-codes. Then we will bound each of the terms individually using Lemmas~\ref{lem:full-matrix-partition-equidistribution} and \ref{lem:full-partition-depth}.

    Since $2\alpha_1 + \alpha_2 < 1 - 2\beta$, we have $\alpha_1 + \beta < 1 - \beta - \alpha_1 - \alpha_2$. Choose $\gamma > 0$ such that \[
    \alpha_1 + \beta < \gamma < 1 - \beta - \alpha_1 - \alpha_2
    \]
    and let $\delta_n = n^{-\gamma}$. We have $w_nh_n = O(n^{\alpha_1 + \alpha_2}) = o(n^{1 - \beta - \gamma}) = o(n\varepsilon_n\delta_n)$ and $h_n\log n = O(n^{\alpha_1 + \alpha_2}) = o(n\varepsilon_n\delta_n)$ as well, so $\delta_n$ satisfies the conditions for Lemmas~\ref{lem:full-matrix-partition-equidistribution}~and~\ref{lem:full-partition-depth}.

    Let $\mathcal{P}, \mathcal{Q}$ be the row and column partitions witnessing the $(w_n, h_n, \varepsilon_n)$-balancedness of $M_n$. We have 
    \begin{align}
              \nonumber&\sum_{f \in \operatorname{Sur}(V, G)}\left|\PP[f(M_n) = 0] - \frac{1}{|G|^{n + u}}\right|   \\
              &\qquad\qquad\leq \sum_{\substack{f \in \operatorname{Sur}(V, G) \\ f \text{ code of distance }n\delta_n}}\left|\PP[f(M_n) = 0] - \frac{1}{|G|^{n + u}}\right| \\
        &\qquad\qquad\qquad +  \sum_{\substack{D > 1 \\ D \mid |G|}}\sum_{\substack{f \in \operatorname{Sur}(V, G) \\ f \text{ of }(\mathcal{P}, \delta_n)\text{-depth }D}}\PP[f(M_n) = 0] \\
        &\qquad\qquad\qquad+ \sum_{\substack{D > 1 \\ D \mid |G|}}\sum_{\substack{f \in \operatorname{Sur}(V, G) \\ f \text{ of }(\mathcal{P}, \delta_n)\text{-depth }D}} \frac{1}{|G|^{n+u}} \\
        &\qquad\qquad\qquad+ \sum_{f \in \operatorname{Hom}(V, G) \setminus \operatorname{Sur}(V, G)} \frac{1}{|G|^{n+u}}
    \end{align}
    
    For notational convenience, we will use $K$ to denote a constant that is allowed to change in each line as long as it remains a constant depending only on $a, u, \alpha_1, \alpha_2, \beta, (h_n)_n, (w_n)_n, G$.

    Wood showed in the proof of \cite[Theorem 2.9]{wood2019matrices} that (4) is bounded above by $Ke^{-n\log 2} = o(1)$. By Lemma~\ref{lem:full-matrix-partition-equidistribution}, we can bound (1): \begin{align*}
     \sum_{\substack{f \in \operatorname{Sur}(V, G) \\ f \text{ code of distance }n\delta_n}}\left|\PP[f(M_n) = 0] - \frac{1}{|G|^{n + u}}\right| \leq |G|^no\left(\frac{1}{|G|^{n+u}}\right) = o(1).
    \end{align*}
    To bound (2) and (3) we use Lemma~\ref{lem:count-depth}: for each $D > 1$, there are at most \[
    K\binom{n}{\lceil \ell(D)n\delta_n\rceil - 1}2^{\ell(D)n\delta_n}|G|^nD^{-n + \ell(D)n\delta_n}
    \]
    maps of $(\mathcal{P}, \delta_n)$-depth $D$. We will start by using this bound from Lemma~\ref{lem:count-depth} to get a slightly weaker bound whose limit behavior is easier to understand. A standard inequality says that $\binom{n}{k} \leq \left(\frac{ne}{k}\right)^k$, so for $\lceil \ell(D)n\delta_n\rceil \geq 2$ (which is the case for $n$ large enough, independent of $D$) \begin{align*}
        \binom{n}{\lceil \ell(D)n\delta_n\rceil - 1} &\leq \left(\frac{ne}{\lceil \ell(D)n\delta_n\rceil - 1}\right)^{\lceil \ell(D)n\delta_n\rceil - 1} \\
        &\leq \left(\frac{2ne}{\ell(D)n\delta_n}\right)^{\ell(D)n\delta_n} \\
        &= \left(\frac{2e}{\ell(D)\delta_n}\right)^{\ell(D)n\delta_n} \\
        &= \exp\left(\ell(D)n\delta_n\left(1 + \log 2 - \log \ell(D) - \log\delta_n\right)\right).
    \end{align*}
    Hence, the number of maps of $(\mathcal{P}, \delta_n)$-depth $D$ is at most \begin{align}\label{eq:bound-depth}
    \nonumber K|G|^nD^{-n}\exp\left(\ell(D)n\delta_n\left(\log \frac{4eD}{\ell(D)} - \log\delta_n\right)\right) &= K|G|^n\exp\left(\ell(D)n\delta_n\left(\log \frac{4eD}{\ell(D)} - \log\delta_n\right) - n\log D\right) \\
    &\leq K|G|^n\exp\left(\ell(|G|)n\delta_n\left(\log \frac{4e|G|}{\ell(|G|)} - \log\delta_n\right) - n\log 2\right) \tag{$*$}
    \end{align}
    Since $\lim_{\delta\to0} \delta\log\delta = 0$ and $\delta_n \to 0$ as $n\to\infty$, for large enough $n$ (depending on $|G|$ and the rate of decay of the sequence $\delta_n$) we have $\ell(|G|)\delta_n\left(\log \frac{4e|G|}{\ell(|G|)} - \log\delta_n\right) \leq \frac{1}{2}\log 2$, which means that for large enough $n$, \begin{align*}
        \sum_{D \mid |G|}\sum_{\substack{f \in \operatorname{Sur}(V, G) \\ f \text{ of }(\mathcal{P}, \delta_n)\text{-depth }D}} \frac{1}{|G|^{n+u}} &\leq \sum_{D \mid |G|} K|G|^{-u}\exp\left(\ell(|G|)n\delta_n\left(\log \frac{4e|G|}{\ell(|G|)} - \log\delta_n\right) - n\log 2\right) \\
        &\leq \sum_{D \mid |G|} K\exp\left(-\frac{\log 2}{2}n\right) \\
        &\leq K\exp\left(-\frac{\log 2}{2}n\right) = o(1),
    \end{align*}
    bounding (3) as desired.

    Finally, we need to bound (2). From Lemma~\ref{lem:full-partition-depth}, we have that if $f$ has $(\mathcal{P}, \delta_n)$-depth $D$, \[
    \PP[f(M_n) = 0] \leq K\exp\left(-\varepsilon_n\frac{n}{w_n}\right)D^n|G|^{-n}, 
    \]
    which, combined with \eqref{eq:bound-depth}, gives \begin{align*}
    \sum_{\substack{f \in \operatorname{Sur}(V, G) \\ f \text{ of }(\mathcal{P}, \delta_n)\text{-depth }D}}\PP[f(M_n) = 0] &\leq K\exp\left(\ell(D)n\delta_n\left(\log \frac{4eD}{\ell(D)} - \log\delta_n\right) -\frac{n\varepsilon_n}{w_n}\right).
    \end{align*}
    The term $\ell(D)n\delta_n\left(\log \frac{4eD}{\ell(D)} - \log\delta_n\right)$ is bounded above by a constant multiple of $n^{1 - \gamma}\log n$ for large enough $n$, whereas $\frac{n\varepsilon_n}{w_n}$ is bounded below by a constant multiple of $n^{1 - \beta - \alpha_1}$. Since $\gamma > \alpha_1 + \beta$, we have $n^{1 - \gamma}\log n = o(n^{1 - \beta - \alpha_1})$, so that the term in the exponent goes to $-\infty$ as $n\to\infty$. Thus, \[
    \sum_{\substack{f \in \operatorname{Sur}(V, G) \\ f \text{ of }(\mathcal{P}, \delta_n)\text{-depth }D}}\PP[f(M_n) = 0] = o(1),
    \]
    giving us a bound on (2). Since each of (1), (2), (3), and (4) is $o(1)$, we obtain the claim of the theorem.
\end{proof}

As explained at the beginning of this section, the results of \cite[Theorem 3.1, Lemma 3.2]{wood2019matrices} imply that Theorem~\ref{thm:matrix-universality} follows from Theorem~\ref{thm:moments}.

\subsection{Weak convergence and proof of Theorem~\ref{thm:intro-matrix-universality}}\label{sect:weak-convergence}

We now give a proof of Theorem~\ref{thm:intro-matrix-universality} from Theorem~\ref{thm:matrix-universality}, expanding on Remark~\ref{rmk:weak-convergence}. This idea is due to Liu and Wood in \cite{liu2018freegroup}, although their topological space consists of profinite groups.

Let $\mathcal{A}$ be the set of isomorphism classes of well-behaved abelian groups. Recall that an abelian group $G$ is well-behaved if $G \otimes \Z/a\Z$ is finite for every positive integer $a$.

Recall that for a positive integer $a$ and finite abelian group $H$ with exponent dividing $a$ we defined \[
U_{a, H} = \{X \text{ well-behaved abelian}\mid X \otimes \Z/a\Z \cong H\} \subset \mathcal{A}.
\]
If $Y \sim \lambda_u$, then $\PP[Y \otimes \Z/a\Z \cong H] = \lambda_u(U_{a, H})$.

The sets $U_{a, H}$ cover $\mathcal{A}$. Moreover, suppose $U_{a, H} \cap U_{a', H'}$ is nonempty. Then it consists of groups $G$ with $G \otimes \Z/a\Z \cong H$ and $G \otimes \Z/a'\Z \cong H'$. Note that both of these quotients are determined by $G \otimes \Z/aa'\Z$ (in fact, considering $\operatorname{lcm}(a, a')$ is enough). So, $U_{a, H} \cap U_{a', H'}$ is covered by the disjoint sets $U_{aa', H''}$ as $H''$ ranges over all finite abelian groups of exponent dividing $aa'$ with $H'' \otimes \Z/a\Z \cong H$ and $H'' \otimes \Z/a'\Z \cong H'$. Thus, the sets $U_{a, H}$ form a basis for a topology on $\mathcal{A}$. From now on, we will view $\mathcal{A}$ as a second-countable topological space with this topology. Since $H \in U_{a, H}$, finite groups are dense in $\mathcal{A}$. We equip $\mathcal{A}$ with the Borel $\sigma$-algebra.

We can give a smaller basis for the same topology by considering only $a = k!$ for positive integers $k$. If $H$ is a finite abelian group with exponent dividing $k!$, let $U_{(k), H} \coloneqq U_{k!, H}$. We observe that if $a \mid a'$, then we can write $U_{a, H}$ as a finite disjoint union of $U_{a', H'_i}$, where $H_i$ ranges over all finite abelian groups of exponent dividing $a'$ such that $H'_i \otimes \Z/a\Z \cong H$. Thus, the sets $U_{(k), H}$ also form a basis for the topology on $\mathcal{A}$.

The sets $U_{(k), H}$ enjoy the property that if $k \leq k'$, then either $U_{(k'), H'} \subseteq U_{(k), H}$ or $U_{(k'), H'} \cap U_{(k), H} = \varnothing$. In particular, every open set in $\mathcal{A}$ is a countable disjoint union of the basic opens $U_{(k), H}$. To see this, suppose $U$ is an open set in $\mathcal{A}$. Then consider the collection of basic opens $U_{(k), H}$ in $U$ and the subcollection of these which are maximal with respect to inclusion. We observe that ascending chains of basic opens in $\mathcal{A}$ stabilize, so every basic open $U_{(k), H}$ in $U$ is contained in a maximal such basic open, and therefore $U$ is the union of its maximal basic opens. Two different such maximal basic opens are necessarily disjoint by the above discussion, so we are done.

We recall the construction of the probability measure $\lambda_u$ on $\mathcal{A}$. We define a random finite abelian $p$-group $Y_p$ as follows: if $B$ is a finite abelian $p$-group, then \[
\PP[Y_p \cong B] = \frac{1}{|B|^u|\operatorname{Aut}(B)|}\prod_{k= u + 1}^\infty (1 - p^{-k}).
\] 
Then set $Y \coloneqq \prod_p Y_p$, where the product ranges over all primes $p$. It follows from \cite[Lemma 3.2]{wood2019matrices} that this indeed determines a probability measure on the countable set $\mathcal{A}_p$ of isomorphism classes of finite abelian $p$-groups with the discrete $\sigma$-algebra. We have a map $\prod_p \mathcal{A}_p \to \mathcal{A}$ given by direct product of groups, where $p$ ranges over all primes. The preimage of $U_{a, H}$ under this map is a countable union of sets of the form $\prod_{p \mid a} \{\widetilde{H}_{p}\} \times \prod_{p \nmid a} \mathcal{A}_p$, where $\widetilde{H}_{p}$ is a finite abelian $p$-group with $\widetilde{H}_{p} \otimes \Z/a\Z$ isomorphic to the $p$-Sylow subgroup of $H$. In particular, the product map $\prod_p \mathcal{A}_p \to \mathcal{A}$ is measurable, and the product $Y = \prod_p Y_p$ indeed defines a random group in $\mathcal{A}$. We let $\lambda_u$ be the associated probability measure on $\mathcal{A}$.

\begin{proof}[Proof of Theorem~\ref{thm:intro-matrix-universality}]
In Theorem~\ref{thm:matrix-universality}, we showed that in the setting of Theorem~\ref{thm:intro-matrix-universality}, we have the convergence statement \[
\lim_{n\to\infty} \PP[\operatorname{coker}(M_n) \in U_{a, H}] = \lim_{n\to\infty} \PP[\operatorname{coker}(M_n) \otimes \Z/a\Z \cong H] = \PP[Y \otimes \Z/a\Z \cong H] = \lambda_u(U_{a, H})
\]
for any positive integer $a$ and finite abelian group $H$ of exponent dividing $a$. We now explain why this implies weak convergence in the topology we have just defined on $\mathcal{A}$. 

Let $U$ be an open set in $\mathcal{A}$. Write $U$ as a countable disjoint union of basic opens $U_i = U_{(k_i), H_i}$: \[
U = \bigsqcup_{i=1}^\infty U_i.
\]
Then by countable additivity and Fatou's lemma, we have \begin{align*}
\lambda_u(U) &= \sum_{i=1}^\infty \lambda_u(U_i) \\
&= \sum_{i=1}^\infty \lim_{n\to\infty} \PP[\operatorname{coker}(M_n) \in U_i] \\
&\leq \liminf_{n\to\infty}\sum_{i=1}^\infty \PP[\operatorname{coker}(M_n) \in U_i] \\
&= \liminf_{n\to\infty}\PP[\operatorname{coker}(M_n) \in U].
\end{align*}
By the Portmanteau theorem, this assertion for every open $U$ is equivalent to weak convergence of the distribution of $\operatorname{coker}(M_n)$ to $\lambda_u$.

On the other hand, weak convergence of the distribution of $\operatorname{coker}(M_n)$ to $\lambda_u$ would imply the convergence statement of Theorem~\ref{thm:matrix-universality} because each $U_{a, H}$ is both open and closed in $\mathcal{A}$. 
\end{proof}

\subsection*{Acknowledgements}
The author was supported by the NSF Graduate Research Fellowship Program, the Caltech Summer Undergraduate Research Fellowship program, and the Samuel P. and Frances Krown SURF Fellowship. The author thanks Melanie Wood and Omer Tamuz for mentorship and Alexander Gorokhovsky, Seth Berman, Sandra O'Neill, and Hoi Nguyen for insightful conversations. The author also thanks Gilyoung Cheong, Yifeng Huang, Hoi Nguyen, Roger Van Peski, Will Sawin, and Melanie Wood for helpful comments on an earlier draft of this manuscript. We especially thank the anonymous referees for comments that inspired a significant improvement to the condition in Theorem~\ref{thm:matrix-universality}.

\printbibliography
\end{document}